\documentclass[11pt]{article}
\usepackage[letterpaper,margin=1in]{geometry}

\usepackage{amsmath}
\usepackage{amssymb}
\usepackage{mathtools}
\usepackage{tikz}
\usepackage{graphicx}
\usepackage{amsthm}
\usepackage{fullpage}

\usepackage[backgroundcolor=white, textwidth=2cm]{todonotes}

\theoremstyle{plain}
\newtheorem{thm}{Theorem}[section]
\newtheorem{cor}[thm]{Corollary}
\newtheorem{lem}[thm]{Lemma}
\newtheorem{obs}[thm]{Observation}

\theoremstyle{definition}
\newtheorem{dfn}[thm]{Definition}
\newtheorem{exm}[thm]{Example}

\newcommand{\versone}{\textup{I}}
\newcommand{\verstwo}{\textup{II}}

\usepackage[colorlinks]{hyperref}
\definecolor{lightblue}{rgb}{0.5,0.5,1.0}
\definecolor{darkred}{rgb}{0.5,0,0}
\definecolor{darkgreen}{rgb}{0,0.5,0}
\definecolor{darkblue}{rgb}{0,0,0.5}

\hypersetup{colorlinks,linkcolor=darkred,filecolor=darkgreen,urlcolor=darkred,citecolor=darkblue}

\begin{document}
\title{A Systematic Study of Isomorphism Invariants of Finite Groups via the Weisfeiler-Leman Dimension
}
\author{Jendrik Brachter and Pascal Schweitzer}

\date{\today}
\maketitle
\begin{abstract}
We investigate the relationship between various isomorphism invariants for finite groups.
Specifically, we use the Weisfeiler-Leman dimension (WL) to characterize, compare and quantify the effectiveness and complexity of invariants for group isomorphism.

It turns out that a surprising number of invariants and characteristic subgroups that are classic to group theory can be detected and identified by a low dimensional Weisfeiler-Leman algorithm. These include the center, the inner automorphism group, the commutator subgroup and the derived series, the abelian radical, the solvable radical, the Fitting group and $\pi$-radicals.
A low dimensional WL algorithm additionally determines the isomorphism type of the socle as well as the factors in the derived series and the upper and lower central series.

We also analyze the behavior of the WL algorithm for group extensions and prove that a low dimensional WL algorithm determines the isomorphism types of the composition factors of a group.

Finally we develop a new tool to define a canonical maximal central decomposition for groups. This allows us to show that the Weisfeiler-Leman dimension of a group is at most one larger than the dimensions of its direct indecomposable factors. In other words the Weisfeiler-Leman dimension increases by at most 1 when taking direct products.
\end{abstract}

\section{Introduction}
Tasks of classifying finite groups up to isomorphism and generating particular classes of finite groups are fundamental and recurring themes in computational group theory. 
Yet, in particular the computational complexity of such problems remains most illusive to date. 

For example, for most orders up to 20.000 the number of non-isomorphic finite groups has been computed and the groups have been exhaustively generated~\cite{Eick2017}. But there are currently 38
notoriously difficult, exceptional cases, for which this information is beyond our current means~(see \cite{Eick2017}). The varying difficulty across different orders is in part caused by the erratic fluctuation of the number of isomorphism classes of finite groups as the order increases. 
This number appears to be closely linked to the multiplicities of the prime factors of the respective order, but even estimating the number of groups of a given order is non-trivial.

 Generation tasks for classes of groups have a long tradition dating back to Cayley~\cite{doi:10.1080/14786445408647421}. Nowadays, there is extensive work on generating particular classes of groups. For example 
there are practically efficient algorithms for the generation of finite nilpotent or finite solvable groups~\cite{MR3197178}. However, the algorithms come without efficient running time guarantees. 

One of the difficulties for a complexity analysis stems from the group isomorphism problem. Indeed, the group isomorphism problem for finite groups stays among the few standard tasks in computational group theory with uncertain complexity. In principle, we desire algorithms with an efficient worst case running time measured in the number of generators through which the groups are given. However, we do not even have algorithms with an efficient worst case running time when measured in the order of the group. In fact the only improvement for the worst case complexity over Tarjan's classic~$n^{\log(n)+O(1)}$ algorithm are~$n^{\frac{1}{c}\cdot \log(n)+O(1)}$ algorithms with a small constant~$c$ depending on the model of computation (randomization, quantum computing etc.) \cite{DBLP:journals/corr/Luks15,DBLP:journals/corr/abs-1304-3935,DBLP:conf/soda/Rosenbaum13}. There is however a nearly-linear time algorithm that solves group isomorphism for most orders~\cite{DietrichWilson}. 

A closely related problem is that of computing isomorphism invariants to distinguish groups. 
Efficiently computable complete invariants are sufficient for general isomorphism testing. However, we do not know efficiently computable complete invariants even for very special cases, such as nilpotent~$p$-groups of class 2. Partial invariants only give incomplete isomorphism tests, but they still find application in generation tasks allowing for heuristic fast pruning~\cite{Eick2017}. 
Given the long history of (algorithmic) group theory, there is an abundance of partial invariants.

Generally the techniques involved in generation and isomorphism computations exploit the existence of various characteristic subgroups classic to group theory. As outlined in~\cite{Eick2017}, these include exploiting the Frattini subgroup~$\Phi(G)$~\cite{DBLP:journals/jsc/BescheE99}, the exponent-$p$-central series~\cite{DBLP:journals/jsc/OBrien90}, characteristic series~\cite{SmithThesis} and similar.

Overall, many of the techniques currently in use are ad-hoc, focused on practical performance, and do not lead to efficient worst case upper bounds for the complexity of the algorithmic problems.  As a consequence, the general picture for finite groups is somewhat chaotic. There is often no structured way of comparing or combining invariants for group isomorphism. E.g., two given invariants may be incomparable in their distinguishing power, making it unclear which invariant to use. Also the required time to evaluate an invariant may be difficult to estimate and can depend significantly on the input group. Even when we are given a class of efficiently computable invariants, it will generally be unclear which invariants to choose or how to efficiently combine their evaluation algorithmically. 

In Summary, we lack the formal means to characterize, compare, or quantify the effectiveness and complexity of invariants for group isomorphism. We therefore propose a systematic study of computationally tractable invariants for finite groups.

For inspiration on how to systematize such a study, we turn to algorithmic finite model theory and specifically descriptive complexity theory. This allows us to characterize the complexity of an invariant by considering a formula within a logic that captures the invariant. A natural choice for a logic from which to choose the formulas is the powerful fixed point logic with counting. Not only can this logic express all polynomial time computable languages on ordered structures~\cite{DBLP:journals/iandc/Immerman86,DBLP:conf/stoc/Vardi82}, but in the context of graphs it has also proven to be an effective tool in comparing invariants (see~\cite{KieferThesis}). As a measure for the complexity of an invariant we can then use the number of variables required to express the invariant in fixed point logic with counting. Crucially there is a corresponding algorithm, the~$k$-dimensional Weisfeiler-Leman algorithm (WL), that (implicitly) simultaneously evaluates  all invariants that are expressible by formulas requiring at most~$k+1$ variables in polynomial time\footnote{For groups there are actually two natural closely related versions of the logic and of the algorithm, $k$-WL$_\versone$ and $k$-WL$_\verstwo$, see Section~\ref{SecTechnicalBasics}.}.

Thus, to enable a quantification and comparison of the complexity of invariants we suggest the Weisfeiler-Leman algorithm. More specifically we suggest to use the Weisfeiler-Leman dimension, which determines how many variables are required to express a given invariant as a formula. 
This gives us a natural and robust framework for studying group invariants. In fact, the~$k$-dimensional Weisfeiler-Leman algorithm is universal for all invariants of the corresponding dimension, resolving the issue of how to combine invariants.  With this approach we also include an abundance of invariants that have not been considered before.
However, it is a priory not clear at all that commonly used invariants can even be captured by the framework, i.e., that they even have bounded WL-dimension.

\paragraph{Contribution}
The first contribution of this paper is to show that a surprising number of isomorphism invariants and subgroups that are classic to group theory can be detected and identified by a low dimensional Weisfeiler-Leman algorithm.
 
Specifically, we show first that for a small value of~$k$, groups not distinguished by~$k$-WL$_\verstwo$ have centers ($k\geq 2$),
inner automorphism groups ($k\geq 4$), derived series~($k\geq 3$), abelian radicals ($k\geq 3$), solvable radicals ($k\geq 2$), fitting groups ($k\geq 3$) and $\pi$-radicals ($k\geq 3$) that are indistinguishable by~$k$-WL$_\verstwo$.
They also have isomorphic socles ($k\geq 5$), stepwise isomorphic factors in the derived series ($k\geq 4$), upper central series ($k\geq 4$), and lower central series ($k\geq 4$). Our techniques regarding characteristic subgroups are fairly general. We thus expect them to be applicable to a large variety of other isomorphism invariants. In particular they should facilitate the analysis of combinations of invariants one might be interested in (such as the Fitting series or the hypercenter).

Beyond these characteristic subgroups, in our second contribution we show that composition factors are incorporated in the invariant computed by a Weisfeiler-Leman algorithm of bounded dimension, in the following sense.

\begin{thm}
 If $k\geq 5$ and $G$ is indistinguishable from $H$ via $k$-WL$_\versone$, then $G$ and $H$ have the same (isomorphism types of) composition factors (with multiplicities).
\end{thm}

The theorem shows that the WL algorithm, which is a purely combinatorial algorithm, can compute group theoretic invariants that do not even appear as a canonical subset of the group. In particular, the composition factors cannot be localized within the group, and at first sight it might not be clear that WL grasps quotient groups.

Our third contribution, having the most technical proof and building on our other results, regards direct products of groups.
Here we consider the decomposition of a group into direct factors. We show that direct products indistinguishable by~$k$-WL must arise from factors that are indistinguishable by~$(k-1)$-WL.

\begin{thm}
	Let $G=G_1\times\dots\times G_d$ be a direct product and $k\geq 5$. If $G$ and~$H$ are not distinguished by~$k$-WL$_\verstwo$ then there are
		direct factors $H_i\leq H$ such that $H=H_1\times\dots\times H_d$ and such that for all~$i$ the groups $G_i$ and~$H_i$ are not distinguished by~$(k-1)$-WL$_\verstwo$.
\end{thm}
In other words, the Weisfeiler-Leman dimension increases by at most 1 when taking direct products. The main difficulty here is that decompositions into direct products are not unique, and thus not definable. These complications arise mainly due to central elements. However we manage to define a canonical maximal central decomposition, that is generally finer than a decomposition into direct factors. We then show that this canonical decomposition is implicitly computed by the WL algorithm.

One way of interpreting our results is that the Weisfeiler-Leman algorithm comprises a unified way of computing all the mentioned invariants and characteristics simultaneously. The dimension can therefore be used to compare the complexity of invariants.

\paragraph{Techniques} To show the various results on characteristic subgroups, we prove a general result on group expressions. It essentially shows that subsets that can be defined by equation systems can be detected by~$k$-WL (see Lemma~\ref{LemGroupExpression}).

The result on composition factors involves a technique that relates~$k$-WL distinguishability of groups to detectable normal subgroups and detectable quotients (Theorem~\ref{MainThm1}).

To deal with direct products, we extend the technique to simultaneously relate chains of subgroups in two indistinguishable groups (Lemma~\ref{ResepctSubgroupChains}). Here we exploit well-known connections of pebble games to Weisfeiler-Leman algorithms. However, the main difficulty regarding our result on decompositions into direct factors is that such decompositions are not unique. In fact in general, a group element cannot be assigned to a direct factor in a well defined sense, making it impossible for WL to detect direct factors. 
For this purpose we develop a new technical tool, component-wise filtrations~(Definition~\ref{dfn:component:wise:filt}), which compensate for the non-uniqueness to extract at least the isomorphism type of the direct factors (Lemma~\ref{lem:semi-abelian-case}). We also exploit the non-commuting graph of the group and show that certain subsets, which we call non-abelian components, can be detected by~$k$-WL (Lemma~\ref{PropertiesNonAbelianComponents}). These non-abelian components lead to a WL-definable maximal central decomposition of every finite group.

\paragraph{Outline} Section~\ref{SecPrelim} provides preliminaries. Section~\ref{SecTechnicalBasics} treats 
WL-refinement in the context of colored groups. 
In Section~\ref{SecWLvsGroupStructure}, we show that invariants generated via WL-refinement fulfill group theoretic closure properties. Section~\ref{SecSpecificInvariants} is an extensive collection of specific structure properties and invariants which Weisfeiler-Leman algorithms detect in finite groups. Finally, in Section~\ref{SecDirectDecompositions} we investigate the ability of WL-refinement to detect direct product decompositions, building on the results of the previous sections.

\paragraph{Further related work}
We should point out that there are various results in the literature on decomposing groups into indecomposable direct factors for various input models of groups. For example there is a polynomial time algorithm to decompose permutation groups into direct products~\cite{WilsonDirect}. Finally, there is a recent algorithm that finds direct product decompositions of permutation groups with factors having disjoint support~\cite{CHANG20221}.
There is also a polynomial time algorithm that  computes direct factors efficiently for groups given by multiplication table~\cite{DBLP:conf/icalp/KayalN09}. Aspects of this algorithm are related to arguments we use for studying the behavior of WL on direct products (see the beginning of Section~\ref{SecDirectDecompositions} for a discussion).

Regarding group isomorphism problems, for isomorphism of Abelian groups a linear time algorithm is known~\cite{DBLP:journals/jcss/Kavitha07} and there are near linear time algorithms for some 
classes of non-abelian groups (e.g,~\cite{DBLP:conf/csr/DasS19}).
Recent directions relate group isomorphism to tensor problems~\cite{DBLP:conf/innovations/GrochowQ21}. The Weisfeiler-Leman algorithm has also been incorporated as a subroutine within other sophisticated group isomorphism algorithms~\cite{DBLP:journals/corr/abs-1905-02518}.

Regarding Weisfeiler-Leman algorithms, the literature is somewhat limited when it comes to groups~\cite{WLonGroups,DBLP:journals/corr/abs-1905-02518} but quite extensive when it comes to graphs. In~\cite{DBLP:journals/jcss/ArvindFKV20}, for example the authors investigate some graph invariants that are captured by the Weisfeiler-Leman algorithm. We refer to~\cite{KieferThesis} for an introduction and an extensive overview over recent results for WL on graphs.

\newpage
 
\section{Preliminaries}\label{SecPrelim}

\paragraph{Sets \& Partitions}
Maximal or minimal sets are always considered with respect to inclusion. We denote multisets as $\{\!\{\dots\}\!\}$. Given disjoint sets $M$ and $N$, their union is $M\uplus N$. An \textit{equipartition} $\mathcal{P}=\{M_1,\dots M_n\}$ of a set $M$ is a partition $M=M_1\uplus\dots\uplus M_n$ such that $|M_i|=|M_1|$ for all $1\leq i\leq n$. A \textit{system of representatives modulo $P$}
is a subset $R\subseteq M$ such that for all $r,s\in R$ and for all $i$ it holds that $(r\in M_i \wedge s\in M_i)\Longrightarrow r=s$. $R$ is \textit{full} if $R$ is a maximal system of representatives. We refer to the $m$-th Cartesian power of $M$ as $M^{(m)}$.

\paragraph{Graphs}
\textit{Graphs} are assumed to be \textit{undirected, simple} and \textit{finite}. We use $V(\Gamma)$ and $E(\Gamma)$ to refer to the vertices or edges of a graph $\Gamma$. For a subset $S\subseteq V(\Gamma)$, let $\Gamma[S]$ denote the subgraph induced by the set $S$. A graph $\Gamma$ is called \textit{bipartite} if we can write $V(\Gamma)=L\uplus R$ such that there are no edges in $\Gamma[L]$ or $\Gamma[R]$. A \textit{matching} on a graph is a collection of disjoint edges. A matching is \textit{perfect} if it covers all vertices. 

\paragraph{Groups}
Groups are assumed to be finite. The symmetric group on $m$ symbols is denoted by $S_m$.
The order of a group element $g\in G$ is the order of the group generated by $g$, i.e., $|g|:=|\langle g\rangle|$. Given a finite set of primes $\pi$, a \textit{$\pi$-group} is a group whose order is only divisible by primes in $\pi$. A group element is called a \textit{$\pi$-element} if it generates a $\pi$-group. For any $d\in\mathbb{Z}$, let $(G)^d:=\langle\{g^d\mid g\in G\}\rangle$.

For a group $G$ and $g,h\in G$, we define the \textit{commutator} $[g,h]:=ghg^{-1}h^{-1}$. We abbreviate the conjugation action to $g^h:=hgh^{-1}$. If $M,N\subseteq G$ we set $[M,N]:=\langle [m,n]\mid m\in M, n\in N\rangle$ and in the special case $M=N=G$ we write $G':=[G,G]$ for the \textit{derived subgroup} of $G$.

Given $m$-tuples of group elements $\bar{g}:=(g_1,\dots,g_m)\in G^{(m)},\bar{h}:=(h_1,\dots,h_m)\in H^{(m)}$, we say 
$\bar{g}$ and $\bar{h}$ have the same \textit{ordered isomorphism type} if there is a group isomorphism
$\varphi:\langle\bar{g}\rangle\to\langle\bar{h}\rangle$ with $\varphi(g_i)=h_i$ for all $1\leq i\leq m$.

\section{Colored Groups \& Weisfeiler-Leman Algorithms}\label{SecTechnicalBasics}
We recapitulate various notions regarding WL-algorithms on groups. For WL on graphs we refer to~\cite{KieferThesis}. For uncolored groups, versions of WL were defined in \cite{WLonGroups}. For our purpose however, we need to formally generalize the concepts to the setting of colored groups. Let us point out that in the setting of colored graphs, colors can be replaced by gadget constructions to obtain uncolored graphs while maintaining the combinatorial properties of the structure. However, for groups it is unclear how to do this. Nevertheless, we will still use colors on groups to restrict the set of possible automorphisms.

\subsection{Colorings on Finite Groups}
Given a natural number $k$ and a finite group $G$, a ($k$-)\textbf{coloring} (over $G$) is just a map $\gamma: G^{(k)}\to\mathcal{C}$ where $\mathcal{C}$ denotes some finite set of colors. A $k$-coloring $\gamma$ partitions $G^{(k)}$ into \textbf{color classes}. We refer to $1$-colorings as \textbf{element-colorings}.

The range $\mathcal{C}$ of target colors is often omitted. Considering two natural numbers $m<k$, a \textbf{$k$-coloring} $\gamma:G^{(k)}\to\mathcal{C}$ induces
an $m$-coloring $\gamma^{(m)}:G^{(m)}\to\mathcal{C}$ via $\gamma^{(m)}((g_1,\dots,g_m)):=\gamma((g_1,\dots,g_m,1,\dots,1))$. To keep our notation simpler we may write $\gamma$ again instead of $\gamma^{(m)}$ and instead of $\gamma^{(1)}$ we use $\gamma^{(G)}$ to emphasize that the coloring is pulled back to group elements.

\begin{dfn}
	A \textbf{colored group} is a group $G$ together with an element-coloring
	$\gamma$ over $G$. Colored groups $(G,\gamma_G)$ and $(H,\gamma_H)$ are \textbf{isomorphic} if there is a
	group isomorphism $\varphi :G\to H$ that respects colors, i.e., $\gamma_H\circ\varphi=\gamma_G$.
\end{dfn}

Given a colored group $(G,\gamma)$ we set Aut$_\gamma(G):=\{ \varphi\in\text{Aut}(G) | \gamma\circ\varphi=\gamma \}$.
\begin{dfn}
	Let $(G,\gamma)$ be a colored group. We say $M\subseteq G$ is \textbf{$\gamma$-induced} if it holds that $\gamma(M)\cap\gamma (G\setminus M)=\emptyset$, i.e., $M$ is a union of $\gamma$-color classes.
\end{dfn}

\subsection{Weisfeiler-Leman Refinement on Colored Groups}
In \cite{WLonGroups}, we introduced three versions of Weisfeiler-Leman algorithms on groups. For the present work it is sufficient to consider two of these versions. The relevant definitions and results are discussed below but we refer to \cite{WLonGroups} for more details.

For $k\geq 2$ we devise a \textbf{Weisfeiler-Leman algorithm of dimension $k$ ($k$-WL)} that takes as input a colored group $(G,\gamma)$ and computes an Aut$_{\gamma}(G)$-invariant coloring on $G^{(k)}$. The algorithm computes an initial coloring from isomorphism invariant properties of $k$-tuples and then iteratively refines color classes until the process stabilizes. The \textbf{stable colorings} arising from~$k$-WL provide (possibly incomplete) polynomial-time non-isomorphism tests.
\vspace{10pt}\newline\underline{Version {\versone} ($k$-WL$_{\versone}$):}
\newline The initial coloring $\chi^{\versone, k}_{\gamma,0}$ is defined via the group's multiplication relation while also taking into account element-colors. Two tuples $\bar{g}:=(g_1,\dots,g_k)$ and $\bar{h}:=(h_1,\dots,h_k)$ obtain the same initial color if and only if for all indices $i,j$ and $m$ between $1$ and $k$ it holds that
\begin{enumerate}
	\item[$\bullet$] $\gamma(g_i)=\gamma(h_i)$,
	\item[$\bullet$] $g_i=g_j\Longleftrightarrow h_i=h_j$,
	\item[$\bullet$] $g_ig_j=g_m\Longleftrightarrow h_ih_j=h_m$.
\end{enumerate} The subsequent refinements are defined iteratively via
\[	
	\chi^{\versone, k}_{\gamma,i+1}(\bar{g}):=\left( \chi^{\versone,k}_{\gamma,i}(\bar{g}), \mathcal{M}(\bar{g}) \right).
\]Here, $\mathcal{M}(\bar{g})$ is the multiset of $k$-tuples of colors given by
\[
	\mathcal{M}(\bar{g}):=\{\!\{ (\chi^{\versone,k}_{\gamma,i}(\bar{g}_{1\leftarrow x}),\dots,\chi^{\versone,k}_{\gamma,i}(\bar{g}_{k\leftarrow x})) \mid x\in G \}\!\},
\]where $\bar{g}_{j\leftarrow x}$ is obtained by replacing the $j$-th entry of $\bar{
g}$ by $x$. 
\vspace{10pt}\newline\underline{Version {\verstwo} ($k$-WL$_{{\verstwo}}$):}
\newline The initial coloring $\chi^{\verstwo,k}_{\gamma,0}$ is defined in terms of colored, ordered isomorphism of tuples. Thus, $\bar{g}=(g_1,\dots,g_k)$ and $\bar{h}=(h_1,\dots,h_k)$ obtain the same initial color if and only if there exists an isomorphism of colored subgroups
\[
	\varphi:\langle \bar{g}\rangle\to\langle \bar{h}\rangle
\]such that $\varphi(g_i)=h_i$ for all $i$. The refinement step is unchanged from Version {\versone}.

Since $G$ is finite, there is a smallest $i$ such that $\chi^{\versone,k}_{\gamma,i}$ and $\chi^{\versone,k}_{\gamma,i+1}$ induce the same color class partition on $G^{(k)}$. At this point color classes become stable
and we obtain the \textbf{stable coloring} $\chi^{\versone,k}_{\gamma}:=\chi^{\versone,k}_{\gamma,i}$. In the same way we define $\chi^{\verstwo,k}_{\gamma}$. For uncolored groups we write $\chi^{\versone,k}_{G}$ and $\chi^{\verstwo,k}_{G}$, respectively.

By definition, the initial colorings are invariant under isomorphisms that respect $\gamma$. This property then holds for the iterated colorings as well. In particular, whenever 
$(G,\gamma_G)$ and $(H,\gamma_H)$ are isomorphic as colored groups, there is a bijection $f:G\to H$ such that
$\chi^{\versone,k}_{\gamma_G}=\chi^{\versone,k}_{\gamma_H}\circ f$ (and the same holds for Version \verstwo). So we obtain a non-isomorphism test by comparing stable colorings computed by $k$-WL$_\versone$ or $k$-WL$_\verstwo$ as follows.

\begin{dfn}
	Let $(G,\gamma_G)$ and $(H,\gamma_H)$ be colored groups. We say $G$ is \textbf{distinguished} from $H$ by $k$-WL$_{\versone}$ if there is \textit{no} bijection $f:G^{(k)}\to H^{(k)}$ with $\chi^{\versone,k}_{\gamma_G}=\chi^{\versone,k}_{\gamma_H}\circ f$. We say $k$-WL$_{\versone}$ \textbf{identifies} $G$ if it distinguishes $G$ from all other (non-isomorphic) groups. We write $G\equiv^{\versone}_k H$ to indicate that $G$ and $H$ are not distinguished by $k$-WL$_{\versone}$. 
	Furthermore, for $m\leq k$, tuples of group elements $\bar{g}\in G^{(m)}$ and $\bar{h}\in H^{(m)}$ are \textbf{distinguished} by $k$-WL$_\versone$ if they obtain different colors in the respective induced $m$-colorings $(\chi^{\versone,k}_{\gamma_G})^{(m)}$ and $(\chi^{\versone,k}_{\gamma_H})^{(m)}$. All definitions also apply to Version $\verstwo$ in the obvious way.	
\end{dfn}

The two versions of Weisfeiler-Leman refinement as introduced above are closely related and we will switch between them whenever convenient.
\begin{lem}{(see \cite{WLonGroups}, Theorem 3.5)}\label{CompareVersions}
	Let $(G,\gamma_G)$ and $(H,\gamma_H)$ be colored groups.	
	\begin{enumerate}
		\item Consider $\bar{g}\in G^{(m)},\bar{h}\in H^{(m)}$ and $k\geq m$. 
		If $\bar{g}$ is distinguished from $\bar{h}$ by $k$-WL$_{\versone}$ then 
		$\bar{g}$ is distinguished from $\bar{h}$ by $k$-WL$_{\verstwo}$. If $\bar{g}$ is distinguished from $				\bar{h}$ by $k$-WL$_{\verstwo}$ then $\bar{g}$ is distinguished from $\bar{h}$ by $(k+1)$-WL$_{\versone}$.
		\item  It holds
		\[
			(G,\gamma_G)\equiv^{\versone}_{k+1} (H,\gamma_H)\Longrightarrow
			(G,\gamma_G)\equiv^{\verstwo}_{k} (H,\gamma_H)\Longrightarrow
			(G,\gamma_G)\equiv^{\versone}_{k} (H,\gamma_H).
		\]
	\end{enumerate}
\end{lem}
The proof is the same as for uncolored groups, see \cite{WLonGroups} for more details. Finally, we note that in \cite{WLonGroups}, we obtain a run time bound of $\mathcal{O}(|G|^{k+1}\log(|G|))$ for both versions of $k$-WL to compute the stable coloring on $G^{(k)}$. The same bound applies to colored groups. In particular, the initial coloring
of $k$-WL$_\verstwo$ is efficiently computable since we only have to compute isomorphism types of $k$-generated subgroups \emph{relative} to a fixed and ordered generating set of size $k$.

\subsection[Bijective k-Pebble Games]{Bijective $k$-Pebble Games}
As with graphs and uncolored groups, Weisfeiler-Leman algorithms on colored groups can be characterized via pebble games and this perspective provides useful tools for our proofs. The characterization closely follows the theory of WL-algorithms on graphs and the reader familiar with these concepts might want to skip to the next section.

For each $k\in\mathbb{N}$ and each version of $k$-WL as introduced above, there is a corresponding \textbf{bijective $k$-pebble game}.
\vspace{10pt}\newline\underline{Bijective 	$k$-pebble game:}
\newline The $k$-pebble game is played on a pair of colored groups $((G,\gamma_G),(H,\gamma_H))$ of equal orders by two players called Spoiler and Duplicator. There are $k$ pairs of pebbles $(p_1,p'_1),\dots,(p_k,p'_k)$ 
and pebbles from different pairs can be distinguished. A state of the game is called a \textbf{configuration} denoted by $[(g_1,\dots,g_k),(h_1,\dots, h_k)]$ with $g_i\in G\uplus\{\perp\}$ and $h_i\in H\uplus\{\perp\}$. The interpretation is that either $g_i\in G$ and $h_i\in H$ which means that the pebble $p_i$ is placed on $g_i$ while $p'_i$ is placed on $h_i$, or $g_i=h_i=\perp$ and then the $i$-th pebble pair is currently not on the board. If we do not specify an initial configuration the game starts on the \textbf{empty configuration} $[(\perp,\dots,\perp),(\perp,\dots,\perp)]$. One round of the game consists of three steps:
\begin{enumerate}
	\item Spoiler picks up a pebble pair $(p_i,p'_i)$.
	\item Duplicator chooses a bijection $f:G\to H$.
	\item Spoiler places $p_i$ on some $g\in G$ and $p'_i$ on $f(g)\in H$.
\end{enumerate}
In each round the winning condition is checked directly after Step $1$. The winning condition is the only difference between the two versions of the game and it is based on the initial coloring of the corresponding version of $(k-1)$-WL.
\newline\underline{Version~$\versone$:}
\newline The pebble pairs apart from $(p_i,p'_i)$ define $(k-1)$-tuples $\widehat{g}$ and $\widehat{h}$ over $G\uplus\{\perp\}$ and $H\uplus\{\perp\}$, respectively.
Spoiler wins if $\chi_{\gamma_G,0}^{\versone,(k-1)}(\widehat{g})\neq \chi_{\gamma_H,0}^{\versone,(k-1)}(\widehat{h})$, where we require that there are no occurrences of $\perp$ in $\widehat{g}$ or $\widehat{h}$. Otherwise the game continues.
\newline\underline{Version~$\verstwo$:}
\newline In this case, Spoiler wins if $\chi_{\gamma_G,0}^{\verstwo,(k-1)}(\widehat{g})\neq \chi_{\gamma_H,0}^{\verstwo,(k-1)}(\widehat{h})$.

We say that Duplicator wins the game if Duplicator has a strategy to keep the game going ad infinitum.

The following correspondence between Weisfeiler-Leman refinement and pebble games is the same as in the uncolored case and  can be proved in complete analogy.
\begin{lem}{(see \cite{WLonGroups}, Theorem 3.2) }\label{AlgoVSGame}
	Let $J\in\{ {\versone},{\verstwo}\}$ and $k\geq 2$. Consider colored groups $(G,\gamma_G)$ and $(H,\gamma_H)$ with $k$-tuples $\bar{g}\in G^{(k)}$ and $\bar{h}\in H^{(k)}$. Then
	$\chi^{J,k}_{\gamma_G}(\bar{g})=\chi^{J,k}_{\gamma_H}(\bar{h})$ if and only if Spoiler has a winning strategy in the configuration $[(g_1,\dots,g_k,\perp),(h_1,\dots,h_k,\perp)]$ in the $(k+1)$-pebble game (Version $J$).  
\end{lem}

\subsection{Induced Colorings \& Refinements}
Before we can start to investigate the relationship between $k$-WL and properties of groups, we collect some useful observations on induced colorings and WL-refinement. The first lemma is well-known in the setting of graphs (or more generally for cellular algebras, see \cite[Theorem 6.1]{MR1674742}) and easily follows for groups as well.

\begin{lem}\label{DetectColorOfCoords}
	Let $J\in\{\versone,\verstwo\}$ and $k\geq 2$. Consider colored groups $(G,\gamma_G)$ and $(H,\gamma_H)$
	with $g_i\in G$ and $h_i\in H$ for $i=1,\dots,k$. Let $\pi\in S_k$. It holds that
	
	\begin{enumerate}
	\item $\chi^{J,k}_{\gamma_G}(g_1\dots,g_k)=\chi^{J,k}_{\gamma_H}(h_1,\dots,h_k)\Longleftrightarrow \chi^{J,k}_{\gamma_G}(g_{\pi(1)}\dots,g_{\pi(k)})=\chi^{J,k}_{\gamma_H}(h_{\pi(1)},\dots,h_{\pi(k)})$ and
	\item $\forall\, 1\leq i\leq k:\chi^{J,k}_{\gamma_G}(g_1\dots,g_k)=\chi^{J,k}_{\gamma_H}(h_1,\dots,h_k)\Longrightarrow \chi^{J,k}_{\gamma_G}(g_i)=\chi^{J,k}_{\gamma_H}(h_i)$.
	\end{enumerate}
\end{lem}

\begin{lem}\label{RespectPartialSubgroups}
	Consider the $k$-pebble game where $k\geq 4$ for Version~$\versone$~and $k\geq 3$ for Version~$\verstwo$ on a pair of groups $(G,H)$. Assume pebble pairs are placed on $(g_1,h_1),\dots,(g_n,h_n)$ where $g_i\in G$, $h_i\in H$ and $0\leq n\leq k-2$. If Duplicator chooses a bijection $f:G\to H$ such that $f(w(g_1,\dots,g_n))\neq w(h_1,\dots,h_n)$ for some word $w$ (allowing inverses), then Spoiler has a winning strategy. (In the case $n=0$ we still require $f(1)=1$).
\end{lem}
\begin{proof}
	By definition of the pebble game, Duplicator chooses the bijection $f$ in Step 2 of the current round and Spoiler previously picked up a pebble pair in Step 1.
	Set $w_G:=w(g_1,\dots,g_n)$ and $w_H:=w(h_1,\dots,h_n)$. In Version~$\verstwo$, Spoiler wins by placing the pebble pair in their hands on $(w_G,f(w_G))$ and then picking up any pebble pair that is currently not on the board (such a pebble pair exists since $n\leq k-2$). Then the respective pebbled tuples in $G$ and $H$ have different ordered isomorphism types.
	So let us consider Version~$\versone$.

	If $n=0$, then $w_G=w_H=1$ and $f(w_G)\neq w_H=1$ by assumption, so Spoiler wins by pebbling $(w_G,f(w_G))$ and then picking up any other pebble pair. Since $f(w_G)^2\neq f(w_G)$ but $w_H^2=w_H$, the resulting configuration is winning for Spoiler.
	
	If $n>0$, Spoiler places the pebble pair in their hands on $(w_G,f(w_G))$ and picks up a pebble pair that is currently not on the board. Duplicator then chooses a new bijection $f_1:G\to H$ and without loss of generality we may assume that $f_1$ maps pebbled group elements accordingly (otherwise we are in the $n=0$ case again). Now either there is some word $w'(x_1,\dots,x_n)$ with $|w'|<|w|$ and $f_1(w'(g_1,\dots,g_n))\neq w'(h_1,\dots,h_n)$ or otherwise we can write $w_G=w'(g_1,\dots,g_n)g_i$ for some $i$, such that $|w'|=|w_G|-1$ and $f(w'(g_1,\dots,g_n))=w'(h_1,\dots,h_n)$. In the first case, Spoiler places the pebble pair in their hands on $(w'(g_1,\dots,g_n),f_1(w'(g_1,\dots,g_n)))$ and picks up the pebble pair on $(w_G,f(w_G))$. In this case we iterate the argument. In the second case, since $k\geq 4$, up to permuting pebble pairs, Spoiler can reach the configuration
	\[
		[(g_i,w'(g_1,\dots,g_n),w_G,\perp,\dots,\perp),(h_i,w'(h_1,\dots,h_n),f(w_G),\perp,\dots,\perp)]
	\] which fulfills the winning condition by construction of $w'$.
	Since the first case can only occur finitely many times, the Lemma follows.
\end{proof}

\begin{lem}\label{DetectTupleCoords}
	Let $J\in\{\versone,\verstwo\}$, $x\in (G,\gamma_G)$ and $y\in (H,\gamma_H)$. Assume that there is some $k$-tuple $t\in G^{(k)}$ with $i$-th entry $t_i=x$ such that for each $t'\in H^{(k)}$ with $t'_i=y$ it holds
	that $\chi^{J,k}_{\gamma_G}(t)\neq \chi^{J,k}_{\gamma_H}(t')$. Then $k$-WL$_J$ distinguishes $x$ from $y$.
\end{lem}
\begin{proof}
	We argue via Lemma~\ref{AlgoVSGame}, i.e., we show that Spoiler has a winning strategy in the~$(k+1)$-pebble game (Version~$J$) with initial configuration $[(x,\perp,\dots,\perp),(y,\perp,\dots,\perp)]$. Duplicator's bijection has to map $x$ to $y$ due to the initial configuration or otherwise Spoiler wins immediately. Spoiler places the $i$-th pebble pair on $(x,y)$. Independent of Duplicators next moves, Spoiler can subsequently pebble the entries of $t$ resulting in a configuration $[(t,\perp),(t',\perp)]$ for some tuple $t'\in H^{(k)}$ with $t'_i=y$. For any such $t'$, the resulting configuration is winning for Spoiler by assumption.
\end{proof}

We say that a coloring $\gamma_2:G^{(k)}\to\mathcal{C}_2$ \textbf{refines} a coloring $\gamma_1:G^{(k)}\to\mathcal{C}_1$, denoted $\gamma_2\preceq\gamma_1$, if each $\gamma_1$-color class is a union of $\gamma_2$-color classes.
\begin{lem}\label{LemPreColorings}
	Let $\gamma_1,\gamma_2$ be colorings on $G$ such that $(\chi^{\versone,k}_{\gamma_1})^{(G)}\preceq\gamma_2\preceq\gamma_1$.
	Then $\chi^{\versone,k}_{\gamma_1}$ and $\chi^{\versone,k}_{\gamma_2}$ induce the same color classes on $G^{(k)}$.
\end{lem}
\begin{proof}
	Fix $k$-tuples $\bar{g},\bar{h}\in G^{(k)}$. Since $\gamma_2\preceq\gamma_1$ holds, we also have $\chi^{\versone,k}_{\gamma_2}\preceq \chi^{\versone,k}_{\gamma_1}$. Assume that $\chi^{\versone,k}_{\gamma_2,0}(\bar{g})\neq\chi^{\versone,k}_{\gamma_2,0}(\bar{h})$ then $\chi^{\versone,k}_{\gamma_1}(\bar{g})\neq \chi^{\versone,k}_{\gamma_1}(\bar{h})$ by Lemma~\ref{DetectColorOfCoords} together with the assumption $(\chi^{\versone,k}_{\gamma_1})^{(G)}\preceq\gamma_2$. So for some $i$ we obtain $\chi^{\versone,k}_{\gamma_1,i}\preceq\chi^{\versone,k}_{\gamma_2,0}$ and therefore $\chi^{\versone,k}_{\gamma_1}\preceq \chi^{\versone,k}_{\gamma_2}$.
\end{proof}

\section{WL-Refinement on Quotient Groups}\label{SecWLvsGroupStructure}
We are now prepared to investigate the interplay between Weisfeiler-Leman refinement and basic group structure, such as subgroups, normal closures or quotients. We introduce the notion of subset selectors to compare pairs of groups in terms of their substructures. 

\begin{dfn}
A \textbf{subset selector} $\mathcal{S}$ is a mapping that associates with each colored group $(G,\gamma)$ a subset $\mathcal{S}(G,\gamma)\subseteq G$.
For each version $J\in \{\versone,\verstwo\}$, a subset selector $\mathcal{S}$ is called \textbf{$k$-WL$_{J}$-detectable}, if
\[
	\chi^{J,k}_{\gamma_G}(\mathcal{S}(G,\gamma_G))\cap \chi^{J,k}_{\gamma_H}(H\setminus\mathcal{S}(H,\gamma_H))=\emptyset,
\] 
for all pairs of colored groups $(G,\gamma_G),(H,\gamma_H)$.
\end{dfn}

When the dependency of $\mathcal{S}(G,\gamma_G)$ on $(G,\gamma_G)$ is clear from the context, we also say that \textbf{$\mathcal{S}(G,\gamma_G)$ is $k$-WL$_{J}$-detectable} (instead of $(G,\gamma)\mapsto S(G,\gamma)$ being detectable).
Examples of $2$-WL$_{J}$-detectable subset selectors include the association of every group with its center ($J=\verstwo$) or the subset selector associating with each group the subset of elements of order $2$. 

We should remark here that in our sense detectable means that the subset of interest is a union of $\chi^{J,k}_{\gamma_G}$-color classes, but we make no statement on how to algorithmically determine which color classes form the set. In that sense it might a priory not be clear that the subset is even computable.

From the definition it follows that if $\mathcal{S}$ is $k$-WL$_{J}$-detectable then $\mathcal{S}(G,\gamma_G)$ is $\chi^{J,k}_{\gamma_G}$-induced in $G$ and hence invariant under Aut$_{\gamma_G}(G)$. Note that $Id(G):=G$ defines a $k$-WL$_{J}$ detectable subset selector for all $k\geq 2$. Furthermore, if $\mathcal{S}$ and $\mathcal{T}$ are $k$-WL$_J$-detectable then so are $\mathcal{S}(G,\gamma_G)\cup \mathcal{T}(G,\gamma_G)$, $\mathcal{S}(G,\gamma_G)\cap \mathcal{T}(G,\gamma_G)$ and $G\setminus \mathcal{S}(G,\gamma_G)$.

\begin{dfn}
	A \textbf{group expression} $\mathcal{E}:=(\mathcal{S}_1,\dots,\mathcal{S}_t;\mathcal{R})$ of length $t$ is a sequence of subset selectors $\mathcal{S}_i$ together with a set $\mathcal{R}$ of words $w(x_1,\dots,x_t)$ over $t$ variables $x_1,\dots,x_t$, allowing inverses. Let $(G,\gamma)$ be a colored group, then a $t$-tuple $(g_1,\dots, g_t)\in G^{(t)}$ is a \textbf{solution} to $\mathcal{E}$ if for each $i$ it holds that $g_i\in \mathcal{S}_i(G,\gamma)$ and for each $w\in\mathcal{R}$ it holds that $w(g_1,\dots,g_t)=1$. Let $Sol_\mathcal{E}(G,\gamma)\subseteq G^{(t)}$ denote the set of all solutions to $\mathcal{E}$ over $(G,\gamma)$.
\end{dfn}

\begin{lem}\label{LemGroupExpression}
	Consider a group expression $\mathcal{E}:=(\mathcal{S}_1,\dots,\mathcal{S}_t;\mathcal{R})$. Let $k\geq t$ and assume that each $\mathcal{S}_i$ is $k$-WL$_\verstwo$-detectable.
	\begin{enumerate}
		\item Let $(G,\gamma_G)$ and $(H,\gamma_H)$ be colored groups. 
		Then all $t$-tuples in $Sol_\mathcal{E}(G,\gamma_G)$ can be distinguished from all $t$-tuples in $H^{(t)}\setminus Sol_\mathcal{E}(H,\gamma_H)$ via $k$-WL$_\verstwo$.
		\item For $1\leq j\leq t$ and colored groups $(G,\gamma)$ define
		 \begin{align*}
		 	Sol^{\exists}_j(G,\gamma) &:=\{x\in G\mid\exists (x_1,\dots,x_t)\in Sol_\mathcal{E}(G,\gamma): x_j=x\} \\
			Sol^{\forall}_j(G,\gamma) &:=\{x\in G\mid(\forall x_i\in \mathcal{S}_i(G,\gamma))_{1\leq i\leq t}: (x_1,\dots,x_{j-1},x,x_{j+1},\dots,x_t)\in Sol_\mathcal{E}(G,\gamma)\}.		
		 \end{align*}
		 Then $Sol^{\exists}_j$ and $Sol^{\forall}_j$ are $k$-WL$_\verstwo$-detectable subset selectors for all $j$.
	\end{enumerate}
	The same holds for $k$-WL$_\versone$, provided $k>t$.
\end{lem}
\begin{proof}
	\begin{enumerate}
		\item Let $\bar{g}=(g_1,\dots,g_t)\in Sol_\mathcal{E}(G,\gamma_G)$ and $\bar{h}=(h_1,\dots,h_t)\in H^{(t)}\setminus Sol_\mathcal{E}(H,\gamma_H)$. First consider the case that there is some word $w\in\mathcal{R}$ such that $w(\bar{h})\neq 1$. Then there does not exist an isomorphism
		between $\langle g_1,\dots,g_t\rangle$ and $\langle h_1,\dots,h_t\rangle$ that maps $g_i$ to $h_i$ for all $i$. Thus, by definition, the $k$-tuples $(g_1,\dots,g_t,1,\dots,1)$ and $(h_1,\dots,h_t,1,\dots,1)$ obtain different initial colors in $k$-WL$_\verstwo$. In the other case, since $\bar{h}\notin Sol_\mathcal{E}(H,\gamma_H)$, there must be some index $j$ such that $h_j\notin \mathcal{S}_j(H,\gamma_H)$. By assumption $\mathcal{S}_j$ is detectable by $k$-WL$_\verstwo$ and $g_j\in \mathcal{S}_j(G,\gamma_G)$, so by definition $\chi_{\gamma_G}^{\verstwo,k}(g_j)\neq \chi_{\gamma_H}^{\verstwo,k}(h_j)$. In particular, $\bar{g}$ and $\bar{h}$ can be distinguished via Lemma~\ref{DetectColorOfCoords}.
		
		The proof for $k$-WL$_\versone$ is almost identical. Note that for Version $\versone$ we assume that $k\geq t+1$. Then, in the first case where the tuples fulfill different relations we use Lemma~\ref{RespectPartialSubgroups} to obtain the result for $k$-WL$_\versone$. The second case can be treated identically for both Version $\versone$ and Version $\verstwo$.
		
		\item Using the statement from Part 1, if we consider $g\in Sol^\exists_j(G,\gamma_G)$ and $h\in H\setminus Sol^\exists_j(H,\gamma_H)$ we are exactly in the situation of Lemma~\ref{DetectTupleCoords} and so $g$ and $h$ can be distinguished via $k$-WL$_\verstwo$, i.e., $Sol^\exists_j$ is detectable. The same argument works for $g\in Sol^\forall_j(G,\gamma_G)$ and $h\in H\setminus Sol^\forall_j(H,\gamma_H)$. The latter is equivalent to the existence
		of $h_i\in \mathcal{S}_i(H,\gamma_H)$ such that
		$(h_1,\dots,h_{j-1},h,h_{j+1},\dots,h_t)\notin Sol_\mathcal{E}(H,\gamma_H)$,
		so in this situation we use Lemma~\ref{DetectTupleCoords} for $H^{(t)}\setminus Sol_\mathcal{E}(H,\gamma_H)$. \qedhere
	\end{enumerate}
\end{proof}

\begin{lem}\label{lem:basicClosure}
	Consider $k$-WL$_{\verstwo}$-detectable subset selectors $\mathcal{S},\mathcal{T}$. Then the following subset selectors are $k$-WL$_{\verstwo}$-detectable:
	\begin{enumerate}
		\item $\mathcal{S}^e$ for each $e\in\mathbb{Z}$, where $\mathcal{S}^e(G,\gamma):=\{ s^e\mid s\in \mathcal{S}(G,\gamma)\}$,
		\item $C_\mathcal{S}(\mathcal{T})$, where $C_\mathcal{S}(\mathcal{T})(G,\gamma):=\{s\in \mathcal{S}(G,\gamma)\mid [s,\mathcal{T}(G,\gamma)]=\{1\}\}$.
	\end{enumerate}
	Provided $k$ is at least $3$, $k$-WL$_{{\verstwo}}$ further detects the following subset selectors:
	\begin{enumerate}
		\item[3.] $\{s_1\dots s_e\mid s_i\in \mathcal{S}(G,\gamma)\}$ for each~$e\in\mathbb{N}$, in particular also $\langle \mathcal{S}(G,\gamma)\rangle$,
		\item[4.] $\{s^{t}:=tst^{-1}\mid s\in \mathcal{S}(G,\gamma), t\in \mathcal{T}(G,\gamma)\}$, in particular also $\langle \mathcal{S}(G,\gamma)^G\rangle$,
		\item[5.] $\mathcal{N}_\mathcal{S}(\mathcal{T})$, where $\mathcal{N}_\mathcal{S}(\mathcal{T})(G,\gamma):=\{s\in \mathcal{S}(G,\gamma)\mid \mathcal{T}(G,\gamma)^s=\mathcal{T}(G,\gamma)\}$,
		\item[6.] $[\mathcal{S},\mathcal{T}]$, where $[\mathcal{S},\mathcal{T}](G):=\langle [s,t]\mid s\in \mathcal{S}(G,\gamma),\ t\in \mathcal{T}(G,\gamma) \rangle$.
	\end{enumerate}
	All statements remain true if we replace Version~$\verstwo$~by Version~$\versone$~everywhere (including the assumptions), provided $k>2$ in Parts 1 and 2 and $k>3$ in Parts 3--6.
\end{lem}
\begin{proof}
	We make repeated use of Lemma~\ref{LemGroupExpression}. Given a group expression 
	$\mathcal{E} :=(\mathcal{S}_1,\dots,\mathcal{S}_t;\mathcal{R})$, define $Sol_j^\exists$ and $Sol^\forall_j$ ($j=1,\dots,t)$ as in Lemma~\ref{LemGroupExpression}.
	\begin{enumerate}
		\item Set $\mathcal{E}=(\mathcal{S},Id;\{x_1^ex^{-1}_2\})$. Then $\mathcal{S}^e=Sol^\exists_2$.
		 \item Set $\mathcal{E}=(\mathcal{S},\mathcal{T};\{[x_1,x_2]\})$. Then $\mathcal{S}(G,\gamma)\cap C_G(\mathcal{T}(G,\gamma))=Sol^\forall_1$.
		 \item We argue by induction over $e$. Let us write $\mathcal{S}^{[e]}(G,\gamma):=\{ g_1\dots g_e\mid g_i\in \mathcal{S}(G,\gamma)\}$.
		 Assume that $\mathcal{S}^{[e]}$ can be detected for $k\geq 3$ and consider $\mathcal{E}=(\mathcal{S}^{[e]},\mathcal{S},Id;\{x_1x_2x^{-1}_3\})$. Then $Sol^\exists_3$ is exactly $\mathcal{S}^{[e+1]}$ and since $\mathcal{S}$ and $\mathcal{S}^{[e]}$ are both detectable, so is $\mathcal{S}^{[e+1]}$. In particular, $\langle\mathcal{S}(G,\gamma)\rangle =\bigcup_e \mathcal{S}^{[e]}(G,\gamma)$ is detectable by $k$-WL$_\verstwo$ as a union of detectable
		 subset selectors.
		 \item Set $\mathcal{E}=(\mathcal{S},\mathcal{T},Id;\{x_2^{-1}x_1x_2x^{-1}_3\})$. The $\mathcal{S}(G,\gamma)$-conjugates of elements in $\mathcal{T}(G,\gamma)$ are precisely $Sol^\exists_3$. Together with Part 3, this shows that the normal closure
		 of $\mathcal{T}$ is detectable by $k$-WL$_\verstwo$ for $k\geq 3$.
		 \item Set $\mathcal{E}=(\mathcal{T},\mathcal{S},Id\setminus\mathcal{T};\{x_2^{-1}x_1x_2x^{-1}_3\})$. If $\mathcal{T}$ is detectable then so is $G\setminus \mathcal{T}(G,\gamma)$ which implies that $G\setminus Sol^\exists_2(G,\gamma)$ is detectable. Finally note that elements of $\mathcal{S}(G,\gamma)$ do not normalize $\mathcal{T}(G,\gamma)$ if and only if they belong to $Sol^\exists_2(G,\gamma)$.
		 \item Set $\mathcal{E}=(\mathcal{T},\mathcal{S},Id;\{[x_1,x_2]x^{-1}_3\})$. Then $Sol^\exists_3$ is the set of all $\mathcal{T}(G,\gamma)$-$\mathcal{S}(G,\gamma)$-commutators and using Part 3, we obtain detectability of the group they generate, namely $[\mathcal{T},\mathcal{S}]$.
	\end{enumerate}
	The analogue statements for $k$-WL$_\versone$ follow from Lemma~\ref{LemGroupExpression} as well, provided $k>t$ in each case.
\end{proof}

We highlight two direct implications of the previous lemma.
\begin{cor}\label{CorNormalClosuresWL}
	 Let $k\geq 3$, $x\in (G,\gamma_G)$ and $y\in (H,\gamma_H)$. If $\chi^{\verstwo,k}_{\gamma_G}(x)=\chi^{\verstwo,k}_{\gamma_H}(y)$ then 
	\[
		\left(\langle x^G\rangle,\gamma_G|_{\langle x^G\rangle}\right) \equiv_{k-1}^\verstwo
	\left(\langle y^H\rangle,\gamma_H|_{\langle y^H\rangle}\right).\
	\] The same holds for Version $\versone$~with $k\geq 4$.
\end{cor}

\begin{cor}\label{normalSub}
	Let $k\geq 3$, $x_i\in G$ and $y_i\in H$ for $i=1,\dots,k-1$. If $\chi^{\verstwo,k}_{G}(x_1,\dots,x_{k-1})= \chi^{\verstwo,k}_{H}(y_1,\dots,y_{k-1})$ then
	$\langle x_1,\dots,x_{k-1}\rangle$ is normal in $G$ if and only if $\langle y_1,\dots,y_{k-1}\rangle$ is normal in $H$.
\end{cor}

In the following example we employ the previous lemma to identify groups as direct products of detectable subgroups via WL-refinement.
\begin{exm}
	Let $G\equiv^{\verstwo}_3 H$ and assume that $G=G_1\times G_2$ with $\chi^{\verstwo,3}_{G}$-induced subgroups $G_i\leq G$. We can use the colors of elements in $G_i$ to define a detectable subset selector
	$K\mapsto K_i:=\{x\in K\mid \chi^{\verstwo,k}_{K}(x)\in \chi^{\verstwo,k}_{G}(G_i)\}$. 
	Since $G\equiv^\verstwo_3 H$, it must hold that $H_i\leq H$ is indistinguishable from $G_i$ via $3$-WL$_\verstwo$. By the previous lemma, $3$-WL$_{{\verstwo}}$ detects $[G_1,G_2]$ and $G_1\cap G_2$, which are both trivial in this case, as well as $\langle G_1,G_2\rangle$, which is equal to $G$. By the definition of detectability, the same conditions must apply for $H_1$ and $H_2$, thus $H=H_1\times H_2$ with $H_i\equiv^\verstwo_3 G_i$.
\end{exm}

\begin{cor}\label{cor:direct_product_of_detectables}
	Let $J=\versone$ and $k\geq 3$, or $k=\verstwo$ and $k\geq 4$. Consider a direct product $G=G_1\times\dots\times G_t$. 
	If each $G_i$ is $\chi^{k,J}_G$-induced and $G\equiv^J_k H$ then $H$ can be decomposed as $H=H_1\times\dots\times H_t$ with $H_i\equiv^J_k G_i$ for all $1\leq i\leq t$.
\end{cor}
In Section~\ref{SecDirectDecompositions} we will discuss the (much harder) case of arbitrary direct decompositions, without the assumption that each direct factor is detectable as a subgroup.

Thus far we considered concrete substructures inside of groups. For the remainder of this section we want to prove that Weisfeiler-Leman algorithms also take into account properties of quotients over detectable subgroups. 

The following combinatorial tool is attributed to van der Waerden, Sperner and K\"onig in~\cite{MIRSKY1966520} (see~\cite{MIRSKY1966520}, Theorem 6.2 and below).
\begin{lem}\label{InducedBijections}
	Let $f:A\to B$ be a bijection of finite sets and let $P:=\{A_1,\dots,A_t\}$ and $Q:=\{B_1,\dots,B_t\}$ be equipartitions of $A$ and $B$, i.e., $m:=|A_i|=|B_j|$ for all $i,j$. Then there exist $m$ full systems of representatives $R_1,\dots, R_m$ of $A$ modulo $P$ such that $A=\biguplus_{i=1}^m R_i$ and
	for all $i$, $f(R_i)$ is a system of representatives modulo $Q$. In particular, for each $i$, $f|_{R_i}$ induces a bijection $P\to Q$.
\end{lem}

\begin{dfn}
	Given a coloring $\gamma:G\to\mathcal{C}$ and a normal subgroup $N\trianglelefteq G$ define the induced \textbf{quotient coloring} $\bar{\gamma}$ on $G/N$ via $\bar{\gamma}(gN):=\{\!\{\gamma(gn)\mid n\in N\}\!\}$.
\end{dfn}

\begin{lem}\label{QuotientColoringLem}
	Let $k\geq 4$ and consider colored groups $(G,\gamma_G)$ and $(H,\gamma_H)$. Assume that there are normal subgroups $N_G\trianglelefteq G$ and $N_H\trianglelefteq H$ which are induced by $\gamma_G$ and $\gamma_H$, respectively, such that $\gamma_G(N_G)=\gamma_H(N_H)$. Then
	\[
		\chi^{\versone,k}_{\bar{\gamma_G}}(g_1N_G,\dots,g_kN_G)\neq \chi^{\versone,k}_{\bar{\gamma_H}}(h_1N_H,\dots,h_kN_H)\Longrightarrow  \chi^{\versone,k}_{\gamma_G}(g_1,\dots,g_k)\neq \chi^{\versone,k}_{\gamma_H}(h_1,\dots,h_k)
	\] for all choices of $g_i\in G$ and $h_i\in H$.
\end{lem}
\begin{proof}
	Using Lemma~\ref{AlgoVSGame}, we argue via the corresponding $(k+1)$-pebble games. The idea is to lift a winning strategy for Spoiler from the game on $(G/N_G,H/N_H)$ to $(G,H)$, where initial configurations are given by $[(g_1N_G,\dots,g_kN_G,\perp),(h_1N_H,\dots,h_kN_H,\perp)]$ and 
	$[(g_1,\dots,g_k,\perp),(h_1,\dots,h_k,\perp)]$, respectively.
	
	If $|N_G|\neq |N_H|$, Duplicator can not even win on $(G,H)$ from the empty configuration, since by assumption
	$N_G$ and $N_H$ are induced by $\gamma_G$ and $\gamma_H$, respectively, and obtain the same colors. We may therefore assume that $|N_G| =|N_H|$. 
	
	By assumption, Spoiler has a winning strategy on the quotients by picking up the $i$-th pebble pair, say. Spoiler picks up the $i$-th pebble pair in the game on $(G,H)$ as well.
	Consider a subsequent Duplicator move $f:G\to H$. By the previous lemma, there are representatives modulo $N_G$, $r^G_1,\dots,r^G_t$ say, that are mapped to a full set of representatives for $H/N_H$, $r^H_i :=f(r^G_i)$ say. The representatives define a bijection $\bar{f}:G/N_G\to H/N_H, r^G_iN_G\mapsto r^H_iN_H$. If $\bar{f}$ is used as a Duplicator move in the game on $(G/N_G,H/N_H)$, then Spoiler has a corresponding winning strategy by placing the $i$-th pebble pair on $(r^G_jN_G,r^H_jN_H)$, say. In the game on $(G,H)$, Spoiler places the $i$-th pebble pair on $(r^G_j,r^H_j)$. The new configurations reached in the two games we consider relate to each other in the same way the initial configurations do: If the $m$-th pebble pair on $(G,H)$ is placed on $(g,h)$ then the $m$-th pebble pair on $(G/N_G,H/N_H)$ is placed on $(gN_G,hN_H)$. Spoiler can iteratively employ this strategy until eventually a configuration $[(g'_1,\dots,g'_{k+1}),(h'_1,\dots h'_{k+1})]$ is reached such that the corresponding configuration on the quotients is winning for Spoiler. Then we are in one of three cases: Either 
$\bar{\gamma_G}(g'_iN_G)\neq\bar{\gamma_H}(h'_iN_H)$ for some $i$, or there exist $i,j$ with $g'_i=g'_j$ modulo $N_G$ and $h'_i\neq h'_j$ modulo $N_H$, or there exist $i,j,m$ with $g'_ig'_j=g'_m$ modulo $N_G$ and $h'_ih'_j\neq h'_m$ modulo $N_H$ (the last two cases could occur with $G$ and $H$ interchanged but this would not affect the proof). 

In the first case, for each bijection $f':g'_iN_G\to h'_iN_H$ there is some $n\in N_G$ with
$\gamma_G(g'_in)\neq\gamma_H(f'(g'_in))$. Thus, if Duplicator maps $g'_iN_G$ to $h'_iN_H$ Spoiler can win in one move by exploiting the mismatched colors. Otherwise Duplicator maps some $g\in g'_iN_G$ to $h\in H\setminus h'_iN_H$ and Spoiler can put a pebble pair $(p_j,p'_j)$ on $(g,h)$ for some $j\neq i$. Then $g'_ig^{-1}\in N_G$ and 
$h'_ih^{-1}\notin N_H$. Since $k\geq 4$, Spoiler can use additional pebbles to successively fix $g^{-1}$ and $g'_ig^{-1}$. Then $g'_ig^{-1}$ must be mapped to $h'_ih^{-1}$ (or otherwise Duplicator can not respect the multiplication relation on the current pebbles) and then $\gamma_G(g'_ig^{-1})\neq \gamma_H(h'_ih^{-1})$.

In the second case $g'_i{g'_j}^{-1}\in N_G$ and $h'_i{h'_j}^{-1}\notin N_H$ and we end up in the same situation we just discussed.

In the third case $g'_ig'_j{g'_m}^{-1}\in N_G$ but $h'_ih'_j{h'_m}^{-1}\notin N_H$. Spoiler can first put a fourth pebble pair on $(g'_ig'_j,h'_ih'_j)$ (as above, if Duplicator does not map $g'_ig'_j$ to $h'_ih'_j$, Spoiler can win immediately) and then we end up in the situation we encountered at the end of the first case again.
\end{proof}

We collect the previous results in our first main theorem which states that whenever $G\equiv^{\versone}_k H$ holds, there is a color preserving correspondence between detected substructures of $G$ and detected substructures of $H$.
\begin{thm}\label{MainThm1}
	Let $k$ be at least $4$.
	\begin{enumerate}
		\item Consider subset selectors $N,U$ and $U/N$ such that for all $(G,\gamma)$ it holds that $N(G,\gamma)\trianglelefteq G$, $N(G,\gamma)\leq U(G,\gamma)$ and $U/N(G/N(G),\bar{\gamma})=U(G)/N(G)$.
		If $N$ and $U/N$ are $k$-WL$_{\versone}$-detectable then so is $U$.
		\item Consider colored groups $(G,\gamma_G)\equiv^{\versone}_k (H,\gamma_H)$. Let $\Psi:G\to H$ be a bijection with
		$(\chi^{\versone,k}_{\gamma_G})^{(G)}\circ\Psi=(\chi^{\versone,k}_{\gamma_H})^{(H)}$. Then $M\subseteq G$ is $\chi^{\versone,k}_{\gamma_G}$-induced if and only if $\Psi(M)\subseteq H$ is $\chi^{\versone,k}_{\gamma_H}$-induced. In this case
		it holds that
		\[\Psi(\langle M\rangle)=\langle \Psi(M)\rangle
		\]
		In particular, if $M$ is a subgroup then so is $\Psi(M)$ and it holds
		\[
			(M,\gamma_G|_M)\equiv^{\versone}_k (\Psi(M),\gamma_H|_{\Psi(M)}).
		\]Additionally, $M$ is normal if and only if $\Psi(M)$ is and then
		it also holds that
		\[
			(G/M,\bar{\gamma_G})\equiv^{\versone}_k (H/\Psi(M),\bar{\gamma_H}).
		\]
	\end{enumerate}
\end{thm}
\begin{proof}
	\begin{enumerate}
		\item Since $(\chi^{\versone,k}_{\gamma})^{(G)}$ is a refinement of $\gamma$,
		the detectability of $U/N$ implies that, with the quotient coloring induced on $G/N(G)$ by $(\chi^{\versone,k}_{\gamma})^{(G)}$, $k$-WL$_\versone$ distinguishes elements $uN(G)$ with $u\in U$ from elements $xN(G)$ with $x\in G\setminus U$.
		Since $N$ is also $k$-WL$_\versone$-detectable, we are in the situation of Lemma~\ref{QuotientColoringLem} with both groups equal to $(G,(\chi^{\versone,k}_{\gamma})^{(G)})$ which implies $\chi^{\versone,k}_{(\chi^{\versone,k}_{\gamma})^{(G)}}(u)\neq \chi^{\versone,k}_{(\chi^{\versone,k}_{\gamma})^{(G)}}(x)$ for all $u\in U$ and $x\in G\setminus U$. Via Lemma~\ref{LemPreColorings} we obtain
	$\chi^{\versone,k}_{\gamma}(u)\neq \chi^{\versone,k}_\gamma(x)$ for all $u\in U$ and $x\in G\setminus U$, so $U$ is detectable by $k$-WL$_\versone$.
	
		\item By definition, $\Psi$ maps $k$-WL$_\versone$-color classes in $G$ to
		$k$-WL$_\versone$-color classes in $H$. Given a $(\chi^{\versone,k}_{\gamma_G})^{(G)}$-induced subset $M\subseteq G$, define a subset selector $\mathcal{S}_M$ by associating with a colored group $K$ the preimage of $(\chi^{\versone,k}_{\gamma_G})^{(G)}(M)$ in $K$. By definition, $\mathcal{S}_M$ is $k$-WL$_\versone$-detectable and $\Psi(M)=\mathcal{S}_M(H,\gamma_H)$. Thus, by Lemma~\ref{lem:basicClosure}, the groups generated by $M$ and $\Psi(M)$ are $k$-WL$_\versone$-detectable and since $(G,\gamma_G)\equiv^{\versone}_k (H,\gamma_H)$, $\langle M\rangle$ and $\langle \Psi(M)\rangle$ must also be indistinguishable via $k$-WL$_\versone$, i.e., $\Psi(\langle M\rangle)=\langle\Psi(M)\rangle$. Furthermore, in the pebble game on $(G,H)$ Spoiler could restrict their moves to $(\langle M\rangle,\langle\Psi(M)\rangle)$, hence $(\langle M\rangle,\gamma_G|_{\langle M\rangle})\equiv^{\versone}_k (\langle\Psi(M)\rangle,\gamma_H|_{\langle \Psi(M)\rangle})$. Analogous arguments work for normal closures instead of generated subgroups. In particular, $M$ is a normal subgroup of $G$ if and only if $\Psi(M)$ is a normal subgroup of $H$.
For the last claim that $(G/M,\bar{\gamma_G}|_{M})\equiv^{\versone}_k (H/\Psi(M),\bar{\gamma_H}|_{\Psi(M)})$, note that otherwise all elements of $(G/M,\bar{\gamma_G}|_{M})$ would be distinguishable from all elements of $(H/\Psi(M),\bar{\gamma_H}|_{\Psi(M)})$ via $k$-WL$_\versone$ and then by Lemma~\ref{QuotientColoringLem}, $(G,\gamma_G)$ and $(H,\gamma_H)$ would be distinguishable as well, contradicting the assumptions. \qedhere
	\end{enumerate}
\end{proof}

\begin{obs}
	For uncolored groups $G$ and $H$, the $k$-WL-induced subgroups 
	are always characteristic subgroups. In this case
	the previous theorem states a correspondence $C_G\equiv_k^\versone C_H$ between $k$-WL induced characteristic subgroups in $G$ and $H$, respectively, as well as a correspondence $G/C_G\equiv_k^\versone H/C_H$ of their respective quotients.
\end{obs}

Finally let us to point out that detectable substructures can be used to limit Duplicator-strategies. This technique will be needed towards the main result of Section~\ref{SecDirectDecompositions}. More precisely, we now show that Spoiler can ``trade off'' one pebble pair to enforce that Duplicator's bijections are simultaneously compatible with detectable substructures in the following sense.  

\begin{lem}\label{ResepctSubgroupChains}
	Let $k\geq 3$ and $J\in\{\versone,\verstwo\}$. Consider groups $G$ and $H$ with $G\equiv^J_k H$, so
	Duplicator has a winning strategy in the $(k+1)$-pebble game (Version $J$).
	Assume $\chi^{J,k}_G$ and $\chi^{J,k}_H$ induce chains of subgroups $G_s\leq \dots\leq G_1\leq G$ and $H_s\leq \dots\leq H_1\leq H$, respectively, such that $\chi^{J,k}_{G}(G_i)=\chi^{J,k}_H(H_i)$ for all $i$.
	Then Duplicator has a winning strategy in the $k$-pebble game (Version $J$) on $(G,H)$ such that each bijection
	$f:G\to H$ chosen by Duplicator's strategy fulfills the following condition:
	\[ 
		\circledast\ \forall x\in G\ \forall i: f(xG_i)=f(x)H_i.
	\]
\end{lem}

\begin{proof}
	Consider a configuration $c:=[(g_1,\dots,g_k),(h_1,\dots,h_k)]$ in the $k$-pebble game and assume that Duplicator has a winning strategy in this configuration, even in the $(k+1)$-pebble game. 
	
	\textbf{Claim 1:} For all $1\leq j\leq k$ there is a bijection $F^{(j)}:G\to H$ with the following properties:
	\begin{enumerate}
		\item after Spoiler picks up the $j$-th pebble pair in configuration $c$, Duplicator can
		play $F^{(j)}$ as a winning move in the $(k+1)$-pebble game.
		\item $\forall x\in G\ \forall i\in\{1,\ldots,s\}: F^{(j)}(xG_i)= F^{(j)}(x)H_i$.
	\end{enumerate}
	We show Claim 1 by induction on~$s$. The case $s=1$ follows in analogy to the induction step. Thus assume that, after Spoiler picks up the $j$-th pebble pair, Duplicator has a winning move $f:G\to H$ in the $(k+1)$-pebble game such that for some fixed $i_0$ it holds that
	\[
		\forall x\in G\ \forall i\leq i_0: f(xG_i)=f(x)H_i.
	\]
	We construct a new bijection $F:G\to H$ that satisfies the desired properties from Claim~1 for all $i\leq i_0+1$. By Lemma~\ref{InducedBijections}, there is a system of representatives $R$ of $G$ modulo $G_{i_0+1}$ such that $f(R)$ is a set of representatives of $H_{i_0+1}$ in $H$. For $g\in G$ let $r_g\in (gG_{i_0+1})\cap R$ denote the representative of $g$. In the $(k+1)$-pebble game, Spoiler could place the $j$-th pebble pair on $(r(g),f(r(g))$ and then pick up the $(k+1)$-th pebble pair. In this case, since $f$ is a winning move for Duplicator, there exists a subsequent winning move $f_g:G\to H$ for Duplicator. We set $F(g):=f_g(g)$. By construction,
	$f_g$ only depends on the fixed representative $r_g$, so $f_g=f_{r_g}$ holds for all $g\in G$.
	
	Intuitively, Spoiler hypothetically marks a coset~$gG_{i_0+1}$ with an extra pebble pair and then~$f_g$ demonstrates how this coset should be mapped as a whole. The proof will now show that piecing together the different maps for all the cosets~$gG_{i_0+1}$ gives a bijection~$F$ with the desired properties.
	
	\paragraph{$F$ is bijective:} 
	Since $\chi^{\verstwo,k}_G(G_{i})=\chi^{\verstwo,k}_H(H_{i})$ 
	and since for all $g\in G$ the bijection $f_g$ is a winning move for Duplicator, it holds that $f_g(gG_{i})=f_g(g)H_{i}	$ for all $i\in \{1,\ldots,s\}$. In particular, for all $x\in G$ it holds that $f_{r_x}(r_xG_{i_0+1})=f_{r_x}(r_x)H_{i_0+1}=f(r_x)H_{i_0+1}$.
	By construction of $F$, for all $x\in G$ it holds that $F(xG_{i_0+1})=f_{r_x}(r_xG_{i_0+1})$. 
	Using the definition of $R$ we note that $r_x\neq r_y$ implies $f(r_x)H_{i_0+1}\cap f(r_y)H_{i_0+1}=\emptyset$.
	Altogether this implies
	\[		
		F(G)=\bigcup_{x\in G}F(xG_{i_0+1})=\bigcup_{x\in G}f_{r_x}(r_xG_{i_0+1})=\bigcup_{x\in G}f(r_x)H_{i_0+1}
	= H.
	\]	
	\paragraph{$F$ respects cosets:} Using the same arguments as in the bijectivity proof together with
	the assertion $\forall x\in G\ \forall i\leq i_0: f(xG_i)=f(x)H_i$ from the induction hypothesis, we obtain
	\[
	 	F(xG_i)=\bigcup_{g\in G_i} f(r_{xg})H_{i_0+1}=F(x)H_i
	\]for all $x\in G$ and for all $i\leq i_0+1$.
	\paragraph{$F$ is a winning move for Duplicator:} Consider configurations of the form 
	\[
		[(g_1,\dots,g_{j-1},g,g_{j+1},\dots,g_k),(h_1,\dots,h_{j-1},F(g),h_{j+1},\dots,h_k))]
	\] with $g\in G$. By construction $F(g)=f_g(g)$ where $f_g$ is a winning move for Duplicator (even with $(k+1)$ pebble pairs) from which the
	configuration above is reachable for Spoiler. Thus, in all configurations reachable from $c$ via $F$, 				Duplicator has a winning strategy. This proves Claim 1. \hfill$\blacktriangleleft$
	
	By Lemma~\ref{AlgoVSGame}, Duplicator has a winning strategy in the $k$-pebble game on the empty configuration
	and by Claim 1, Duplicator can always choose bijections according to condition $\circledast$.
\end{proof}

\section{WL-dimension of certain isomorphism invariants}\label{SecSpecificInvariants}
We turn our attention to combinatorial invariants and algebraic properties of groups in connection to WL. Recall that $k$-WL$_\verstwo$~explicitly computes isomorphism types of $k$-generated subgroups.
\begin{lem}
	For $k\geq 2$, $k$-WL$_{\verstwo}$ identifies all finite $k$-generated groups.
\end{lem}

\begin{lem}\label{lem:order_and_powers}
	Let $G$ be a finite group and $x,y\in G$. If $\chi^{\verstwo,2}_G(x)=\chi^{\verstwo,2}_G(y)$ then
	$|x|=|y|$ and for all $e\in\mathbb{N}$ the element $x$ is an $e$-th power in $G$ if and only if $y$ is.
\end{lem}
\begin{proof}
	The order of $x$ and $y$ is implicitly contained in the isomorphism type, i.e., the initial coloring, of $(x,1)$ and $(y,1)$. The second claim is a direct consequence of the fact that for all $e$, the set of $e$-th powers in $G$ is a union of stable color classes under $2$-WL$_\verstwo$ which was shown in Lemma~\ref{lem:basicClosure}.
\end{proof}

By the fundamental theorem of abelian groups, finite abelian groups are uniquely determined up to isomorphism by their multisets of element orders. Together with the observation that $2$-WL distinguishes elements based on orders of centralizers (Lemma~\ref{lem:basicClosure}), we can now state the follwoing well-known result in terms of WL-refinement.
\begin{cor}\label{cor:abelian-identified}
	$2$-WL$_{{\verstwo}}$ identifies all finite abelian groups.
\end{cor}

\subsection{Derived \& Central Series}
We group the investigated properties in a thematic manner, beginning with central subgroups and commutators. We briefly introduce several fundamental notions from the theory of groups in what follows. A detailed treatment can be found in \cite{hall1999theory} for example. 
\begin{lem}\label{LemCenterDetectable}
	Consider groups~$G,G_1$ and~$G_2$.
	\begin{enumerate}
		\item If $G_1\equiv_2^\verstwo G_2$ then there is a bijection $f:G_1\to G_2$ such that
	for all $g\in G_1$ it holds that $|C_{G_1}(g)|=|C_{G_2}(f(g))|$.
		
		\item $Z(G)$ is~$2$-WL$_{\verstwo}$-detectable.
	\end{enumerate}
\end{lem}

Every $g\in G$ induces an \textbf{inner} automorphism $\kappa_g$ of $G$ via $\kappa_g(h):=ghg^{-1}$. This defines the \textbf{inner automorphism group} $Inn(G)=\{\kappa_g\mid g\in G\}\cong G/Z(G)$.

\begin{cor}
	For any $k\geq 4$ it holds that $G\equiv^{\versone}_k H\Longrightarrow Inn(G)\equiv^{\versone}_k Inn(H)$.
\end{cor}
\begin{proof}
	The result follows from the previous lemma together with Theorem~\ref{MainThm1}.
\end{proof}

\begin{cor}\label{cor:derived_detectable}
	For $k\geq 3$,~$G'\coloneqq [G,G]$ is $k$-WL$_{{\verstwo}}$-detectable.
\end{cor}
\begin{proof}
	This is Part 6 of Lemma~\ref{lem:basicClosure} for $M=N=G$.
\end{proof}

Let us point out that $k\geq 3$ is a necessary requirement in the previous corollary as computations on SmallGroup(128,171) and SmallGroup(128,1122) from the Small Groups Library in GAP (\cite{GAP4}) show. 

The \textbf{derived series} of a group $G$ is defined as follows.
Set $G_{(0)}:=G$ and for all $i>0$ let $G_{(i)}:=\left(G_{(i-1)}\right)'$. This defines a chain of characteristic subgroups $G=G_{(0)}\geq G_{(1)}\geq\dots\geq G_{(t)}=G_{(t+1)}$ for some $t\geq 0$ (we assume $G$ to be finite), the derived series of $G$. Furthermore, let $G_{(\infty)}:=G_{(t)}$ denote the stable term of the derived series.

\begin{lem}\label{lem:derived_series_detectable}
	Let $k\geq 4$ and $G\equiv^{\versone}_k H$, then $G/G_{(\infty)}$ and $H/H_{(\infty)}$ have step-wise isomorphic derived series, i.e., $G_{(i)}/G_{(i+1)}\cong H_{(i)}/H_{(i+1)}$ for all $i$. Additionally $G_{(i)}\equiv^{\versone}_k H_{(i)}$ holds for all $i$.
\end{lem}
\begin{proof}
	We show the following: If $G$ and $H$ are not distinguished then $G/G'\cong H/H'$ and $G'$ is not distinguished from $H'$. Then the claim follows by induction. By the previous corollary, the commutator subgroup is detectable by Version $1$ WL for $k\geq 4$. Theorem~\ref{MainThm1} together with $G\equiv^{\versone}_k H$ implies $G'\equiv^{\versone}_k H'$ and $G/G'\equiv^{\versone}_k H/H'$. The latter is actually equivalent to
$G/G'\cong H/H'$ by Corollary~\ref{cor:abelian-identified}, since $G/G'$ and $H/H'$ are abelian.
\end{proof}

\begin{cor}\label{cor:solvability_detectable}
	For $k\geq 4$, $k$-WL$_{{\versone}}$ distinguishes solvable from non-solvable groups.
\end{cor}

Similar arguments work for the lower and upper central series. Let $Z_0:=G$ and $Z_{i+1}:=[Z_i,G]$. Then $Z'_i\leq [Z_i,G]=Z_{i+1}$ and thus $Z_i/Z_{i+1}$ is again abelian. Define $Z_\infty$ as the stable term of this series, the \textbf{lower central series} of $G$. Then $G$ is nilpotent if and only if $Z_\infty=\{1\}$ and in this case the nilpotency class of $G$ is the minimal $c$ such that $Z_\infty(G)=Z_{c}(G)$.

\begin{lem}
	Let $k\geq 4$ and $G\equiv^{\versone}_k H$, then 
	$Z_i(G)/Z_{i+1}(G)\cong Z_i(H)/Z_{i+1}(H)$ for all $i$. Additionally $Z_{i}(G)\equiv^{\versone}_k Z_{i}(H)$ and $G/Z_{i}(G)\equiv^{\versone}_k H/Z_{i}(H)$ hold for all $i$.
	In particular, $G$ and $H$ have the same nilpotency class (including the possibility that both $G$ and $H$ are not nilpotent).
\end{lem}
We omit the proof since it is analogous to the proof for the derived series.
The \textbf{upper central series} is defined implicitly via $Z^0:=\{1\},\ Z^1:=Z(G)$ and $Z^{i+1}/Z^i:=Z(G/Z^i)$. 
\begin{lem}\label{lem:upper_central_detectable}
	Let $k\geq 4$ and $G\equiv^{\versone}_k H$, then 
	$Z^{i+1}(G)/Z^{i}(G)\cong Z^{i+1}(H)/Z^{i}(H)$ for all $i$. Additionally $Z^{i}(G)\equiv^{\versone}_k Z^{i}(H)$ and $G/Z^{i}(G)\equiv^{\versone}_k H/Z^{i}(H)$ hold for all $i$.
\end{lem}
\begin{proof}
	By definition, $x\in G$ is in $Z^{i+1}(G)$ if and only if for all $y\in G$ it holds $[x,y]\in Z^i(G)$. 
	If $Z^i(G)$ is detected by $k$-WL$_\versone$ and $k\geq 4$ then $Z^{i+1}(G)$ is detected as well. To see this, consider the group expression $(G,G,G\setminus Z^i(G);\mathcal{R}:=\{[x_1,x_2]=x_3\})$. Then $Z^{i+1}(G)$
	is the complement of $T^\exists_1$ in $G$, where $T^\exists_1$ is defined as in Lemma~\ref{LemGroupExpression} and in particular this set is detectable. The indistinguishablility of $G$ and $H$ now inductively implies $Z^{i}(G)\equiv^{\versone}_k Z^{i}(H)$ for all $i$. Theorem~\ref{MainThm1} gives us $Z^{i+1}(G)/Z^{i}(G)\equiv^{\versone}_k Z^{i+1}(H)/Z^{i}(H)$ for all $i$ which can be replaced by isomorphism since these quotients are abelian by definition.
\end{proof}

\subsection{Radicals}
Let $\mathcal{F}$ be a class of finite groups that is closed under isomorphism and normal products (i.e., if $N_1,N_2\trianglelefteq G$ belong to $\mathcal{F}$ then so does $N_1N_2\leq G$). Furthermore let $G$ be an arbitrary finite group. Then the \textbf{$\mathcal{F}$-radical} $\mathcal{O}_\mathcal{F}(G)$ of $G$ is defined as the subgroup generated by normal subgroups of $G$ belonging to $\mathcal{F}$, i.e., the largest normal $\mathcal{F}$-subgroup in $G$. We consider the following explicit standard examples: 
\begin{enumerate}
	\item the abelian radical $\mathcal{A}(G)$, where $\mathcal{F}$ is the class of abelian groups
	\item the $\pi$-radical $\mathcal{O}_\pi(G)$, where $\pi$ is a collection of primes and
	$\mathcal{F}$ is the class of $\pi$-groups (groups whose order is divisible by primes in $\pi$ only)
	\item the nilpotent radical, also known as the \textbf{Fitting subgroup}, denoted $Fit(G)$, where
	where $\mathcal{F}$ is the class of nilpotent groups
	\item the solvable radical $\mathcal{R}(G)$, where $\mathcal{F}$ is the class of solvable groups
\end{enumerate} 
 
We work in the general setting first and later come back to the examples from above.
\begin{lem}
	Assume $\mathcal{F}$ is closed under normal subgroups. Then $\mathcal{O}_\mathcal{F}(G)$ is the set
	of all elements whose normal closure belongs to $\mathcal{F}$.
\end{lem}
\begin{proof}
	By definition $\mathcal{O}_\mathcal{F}(G)$ is the unique largest normal $\mathcal{F}$-subgroup in $G$.
	Thus, if $x\in \mathcal{O}_\mathcal{F}(G)$ then the normal closure of $x$ is a normal subgroup contained in the radical
	and hence belongs to $\mathcal{F}$ since $\mathcal{F}$ is assumed to be closed under normal subgroups. On the other hand it is always the case that elements with 
	normal closure in $\mathcal{F}$ contribute to $\mathcal{O}_\mathcal{F}(G)$, since
	this is the unique largest normal $\mathcal{F}$-subgroup in $G$. 
\end{proof}

\begin{lem}
	Let $k\geq 3$. If $\mathcal{F}$ is closed under normal subgroups and $(k-1)$-WL$_{{\verstwo}}$ distinguishes
	$\mathcal{F}$-groups from all other groups, then $k$-WL$_{{\verstwo}}$ detects $\mathcal{O}_\mathcal{F}(G)$ in $G$.
\end{lem}
\begin{proof}
	In Corollary~\ref{CorNormalClosuresWL} we proved that $k$-WL$_\verstwo$ can distinguish elements based on $(k-1)$-dimensional properties of their normal closures. Then by the assumptions above and the previous lemma, $k$-WL$_{\verstwo}$ detects the set of group elements whose normal closure belongs to $\mathcal{F}$ which is precisely the radical in this case.
\end{proof}
Also note that in all examples we consider above, $\mathcal{F}$ is indeed closed under normal subgroups.

In the previous section we showed that $2$-WL$_\verstwo$ identifies all abelian groups
and distinguishes $\pi$-groups from other groups for fixed $\pi$.
\begin{cor}
	$3$-WL$_{{\verstwo}}$ detects $\mathcal{A}(G)$ as well as $\mathcal{O}_\pi(G)$, where $\pi$ is an arbitrary collection of primes.
\end{cor}

Regarding the Fitting subgroup, recall that finite nilpotent groups are direct products of $p$-groups (see e.g.\,\cite{zassenhaus1999theory}) and thus the nilpotent radical of a finite group, i.e., $Fit(G)$, is the largest normal subgroup that is a direct product of $p$-groups. By definition this means 
\[
	Fit(G)=\underset{p\ \mid\ |G|}{\bigtimes} \mathcal{O}_p(G).
\]
The detectablility of the Fitting subgroup therefore follows from the discussion of $\pi$-radicals for $\pi:=\{p\}$ and Corollary~\ref{cor:direct_product_of_detectables}.
\begin{cor}
	$Fit(G)$ is detectable by $3$-WL$_{{\verstwo}}$.
\end{cor}

There is another well-known characterization of the $\pi$-radical as the intersection of all maximal $\pi$-subgroups of $G$. Using WL-refinement, it is even possible to consider arbitrary intersections of maximal $\pi$-subgroups.

\begin{lem}
	Let $g\in G$ and $h\in H$ be two $\pi$-elements and let $S_g\leq G$ and $S_h\leq H$ be the intersections of all maximal $\pi$-subgroups containing $g$ and $h$, respectively. If $\chi^{\verstwo,k}_G(g)=\chi^{\verstwo,k}_H(h)$ then $S_g\equiv^\verstwo_{k-1} S_h$.
\end{lem}
\begin{proof}
		Let $\Pi_K$ be the set of maximal $\pi$-subgroups of $K$ for $K\in \{G,H\}$. For $\pi$-elements $g\in G$
		consider $M_g:=\{x\in G\mid \langle x,g\rangle\text{ is a $\pi$-group}\}$ and similarly define $M_h$ for $h\in H$.
		Then for all $\pi$-elements $x\in G$ it holds that $M_x=\bigcup_{P\in\Pi_G,x\in P} P$ and $S_x=\{g\in G\mid M_x\subseteq M_g\}$ and the same holds for $\pi$ elements in $H$. If $x$ is not distinguished from $y$ by $k$-WL$_\verstwo$ then by Lemma~\ref{AlgoVSGame}, Duplicator has a winning strategy in the corresponding $(k+1)$-pebble game on $(G,H)$ starting in the configuration
		$[(x,\perp^k),(y,\perp^k)]$. As long as the pair $(x,y)$ is pebbled, Duplicator has to map 
		$M_x$ to $M_y$ or otherwise Spoiler can immediately reach a configuration where a $\pi$-group is matched with a non-$\pi$-group and win. But then the same holds for $S_x$ and $S_y$ by the way we characterized these sets above. Since Spoiler can leave the first pebble pair on $(x,y)$ and still use the remaining $k$ pebble pairs freely the claim follows via Lemma~\ref{AlgoVSGame}.
\end{proof}

Lastly, the solvable radical can be handled using the detectability of the derived series we proved earlier in this section (Corollary~\ref{cor:solvability_detectable}). We can slightly improve the bound on the WL-dimension via the following "Thompson-like" characterization of the solvable radical.

\begin{thm}[\cite{MR2228653}, Theorem 1.1]
	Let $G$ be a finite group and $\mathcal{R}(G)$ the solvable radical of $G$. An element $g\in G$ belongs to $\mathcal{R}(G)$ if and only if for every $h\in G$ the subgroup generated by $g$ and $h$ is solvable.
\end{thm}

\begin{cor}
	$\mathcal{R}(G)$ is detectable by $2$-WL$_{{\verstwo}}$.
\end{cor}
\begin{proof}
	By the previous theorem, membership to $R(G)$ can be decided in terms of isomorphism types of subgroups generated by pairs $(x,y)$ for fixed $x\in G$.
\end{proof}

\subsection{Simple Groups \& Composition Factors}

Lastly, we consider simple and minimal normal subgroups. Recall that finite simple and almost simple groups can be generated with $2$ and $3$ elements, respectively~\cite{MR1358262}.

\begin{lem}
	$2$-WL$_\verstwo$~identifies finite simple groups.~$3$-WL$_\verstwo$~identifies finite almost simple groups.
\end{lem}
In the case of simple groups there is a stronger result, stating that simple groups are uniquely identified among all groups up to isomorphism by their order and the orders of their elements \cite{SimpleBySpectrum}.

\begin{lem}\label{SimpleNormals}
	Let $G=G_1\times\dots\times G_k\times A$ where for all $i$, $G_i$ is non-abelian simple and $A$ is abelian. Let $S\trianglelefteq G$ be non-abelian simple, then~$S=\{1\}\times\dots\times\{1\}\times G_i\times\{1\}\times\dots\times\{1\}
	$ for some~$i$.
\end{lem}
\begin{proof}
	Otherwise~\cite[Theorem 4.3A]{MR1409812} would imply~$S\leq C_G(G_1\times\dots\times G_k\times\{1\})=Z(G)$ which contradicts the fact that~$S$ is non-abelian.
\end{proof}

Let us recall that a group is called \textbf{characteristically simple} if it does not contain any characteristic subgroups. Finite characteristically simple groups are precisely the finite direct products of isomorphic simple groups~\cite{zassenhaus1999theory}.

\begin{lem}
	Let $G$ be finite and characteristically simple then $G\equiv^{{\verstwo}}_3 H$ if and only if $G\cong H$. In other words, $3$-WL$_\verstwo$ identifies characteristically simple groups.
\end{lem}
\begin{proof}
	By assumption $G\cong T^m$ for some simple group $T$. If $T$ is abelian then so is $G$ and we already discussed the abelian case. Thus assume that $T$ is non-abelian. Since simple groups are $2$-generated, Lemma~\ref{normalSub} shows that $3$-WL$_\verstwo$ detects the set $M_G$ of all pairs $(g_1,g_2)\in G^{(2)}$ that generate a normal subgroup isomorphic to $T$. Let $E_G:=\{g\in G\mid \exists x\in G: (g,x)\in M_G\}$.
	By Lemma~\ref{SimpleNormals}, the normal subgroups of $G$ isomorphic to $T$ are exactly the $m$ different copies of $T$ defining $G$, so $|E_G|=m|\{t\in T\mid \exists x\in T: \langle t,x\rangle=T\}|$. If $G\equiv_3 H$, the
	 corresponding sets $E_H\subseteq H$ and $M_H\subseteq H^{(2)}$, which are defined in the same way as $E_G$ and $M_G$, must be indistinguishable from $E_G$ and $M_G$ via $3$-WL$_\verstwo$. Thus $H$ contains at least $m$ different normal subgroups isomorphic to $T$. Due to simplicity of $T$ they must intersect trivially and centralize each other (given distinct normal subgroups $T_1,T_2\cong T$, $T_1\cap T_2$ is normal in $T_i$ and thus $[T_1,T_2]\leq T_1\cap T_2=\{1\}$) and considering $|H|=|G|=|T|^m$ we obtain $H\cong T^m$.
\end{proof}

\begin{lem}\label{LemDirectProdSimple}
	For $k\geq 3$, $k$-WL$_{{\verstwo}}$ identifies finite direct products of simple groups.
	More precisely, consider $G=T_1^{m_1}\times\dots\times T_r^{m_r}$ with pairwise non-isomorphic simple groups $T_1,\dots,T_r$. Then
	$\{1\}\times\dots\times\{1\}\times T_i^{m_i}\times\{1\}\times\dots\times\{1\}$ is $3$-WL$_\verstwo$-detectable for all $i$.
\end{lem}
\begin{proof}
	As in the previous proof, $3$-WL$_\verstwo$ can distinguish pairs that generate a normal subgroup isomorphic to some fixed non-abelian simple group from other pairs. Together with Lemma~\ref{SimpleNormals} this implies that direct factors of
the form $\{1\}\times\dots\times\{1\}\times T_i^{m_i}\times\{1\}\times\dots\times\{1\}$ for non-abelian simple $T_i$ are detected in $G$ by $3$-WL$_\verstwo$. For abelian $T_i^{m_i}$ note that $T_i\cong C_{p_i}$ for some prime $p_i$ and then $\{1\}\times\dots\times\{1\}\times T_i^{m_i}\times\{1\}\times\dots\times\{1\}$ coincides with the detected set of central $p_i$-elements in $G$.
In conclusion, $G$ is identified by $3$-WL$_\verstwo$ as a direct product of detected subgroups (see Corollary~\ref{cor:direct_product_of_detectables}) which are themselves identified by $3$-WL$_\verstwo$ by the previous lemma.
\end{proof}

From our observations on the WL-detectability of the derived series we can deduce that $4$-WL$_{{\verstwo}}$ implicitly distinguishes solvable groups according to composition factors and their respective multiplicities. We show that this is true for non-solvable groups as well.

Recall that the \textbf{socle} $soc(G)$ of a group $G$ is the subgroup generated by all minimal normal subgroups of $G$. In the case of finite $G$, the socle is a direct product of minimal normal subgroups and minimal normal subgroups are characteristically simple. So for finite groups $G$ we can write
\[ 
	soc(G)=N_1\times\dots\times N_t
\] with each $N_i$ a minimal normal subgroups of $G$, $N_i\cong S_i^{m_i}$ is a direct power of some simple group $S_i$ and we may assume that the simple groups are pairwise non-isomorphic~\cite[Section 4.3]{MR1409812}. 

\begin{lem}
	Let $G$ be a finite group and write $soc(G)=N_1\times\dots\times N_t$ as above. Then for all $i$, $4$-WL$_{{\verstwo}}$ detects $N_i$ in $G$ and in particular $soc(G)$ is detected as well.
\end{lem}
\begin{proof}
	We first show that $4$-WL$_\verstwo$ detects the set of elements $x\in G$ whose normal closures are minimal
	normal subgroups of $G$. Assume $x$ has a normal closure $N_x$ that is minimal normal in $G$ and suppose $y\in G$ has non-minimal normal closure $N_y$.
	By Corollary~\ref{CorNormalClosuresWL}, if $x$ is not distinguished from $y$ then $N_x\equiv^{\verstwo}_3 N_y$. But by minimality of $N_x$, for each $x'\in N_x$ it holds $\langle x'^G\rangle=N_x$ while there is some $y'\in N_y$ with
	$\langle y'^G\rangle\lneq N_y$. So by Corollary~\ref{CorNormalClosuresWL} $y'$ is distinguished from each $x'\in N_x$ by $3$-WL$_\verstwo$.
 Thus, $4$-WL$_\verstwo$ identifies elements whose normal closures are minimal normal subgroups and together they generate $soc(G)$, so the latter is detected as well according to Lemma~\ref{lem:basicClosure}.
	
	The claim then follows from Lemma~\ref{LemDirectProdSimple} together with the fact that $soc(G)$ is a direct product of simple groups.
\end{proof}

\begin{thm}
 Let $k\geq 5$ and $G\equiv^{{\versone}}_k H$ then $G$ and $H$ have the same composition factors (with multiplicities).
\end{thm}
\begin{proof}
	Together with Lemma~\ref{CompareVersions} the previous lemma implies that $soc(G)$ and $soc(H)$ are detected in $G$ and $H$, respectively. By Theorem~\ref{MainThm1} we obtain $soc(G)\equiv^{{\versone}}_k soc(H)$
	as well as $G/soc(G)\equiv^{{\versone}}_k H/soc(H)$. Then first $soc(G)$ and $soc(H)$ have the same composition factors (with multiplicities) by Lemma~\ref{LemDirectProdSimple} and inductively the same holds for $G/soc(G)$ and $H/soc(H)$. Note that in the base case it holds $G=soc(G)$ and $H=soc(H)$ so in this case we are done.
	Now by normality of $soc(G)$, the composition factors of $G$ are precisely the composition factors of $soc(G)$ together with the composition factors of $G/soc(G)$ (in each case considered with multiplicities) and the same holds for $H$, so the claim follows inductively.
\end{proof}

\section{WL-Refinement and Direct Products}\label{SecDirectDecompositions}
In this final section we study the detectability of direct product structures in finite groups. The section is organized similar to~\cite{DBLP:conf/icalp/KayalN09}, in the sense that we first consider direct products where one factor is an abelian group (the semiabelian case) and reduce to these the general case later on. There is also a similarity in the way the direct factors are computed modulo central elements. However, a crucial difference between our setting and the one in~\cite{DBLP:conf/icalp/KayalN09} is that in the latter computations can be executed as long as they are efficient, where in our case, we are analyzing a fixed algorithm that cannot make non-canonical choices.

\begin{dfn}
	Given groups $G_1$ and $G_2$, central subgroups $Z_1\leq Z(G_1)$, $Z_2\leq Z(G_2)$ and an isomorphism $\varphi:Z_1 \to Z_2$, we can form the \textbf{central product} of $G_1$ and $G_2$ with respect to $\varphi$ via
	\[
		G_1\times_{\varphi} G_2 := G_1\times G_2/\{(g,\varphi(g^{-1}))\mid g\in Z_1\}.
	\]A group $G$ is the \textbf{(internal) central product} of subgroups $G_1,G_2\leq G$, if it holds that $G=\langle G_1,G_2\rangle$ and $[G_1,G_2]=\{1\}$.
\end{dfn}

Our main difficulty is that a group can admit several inherently different central decompositions. In contrast to that recall that indecomposable \textit{direct} decompositions are unique in the following sense.

\begin{lem}\label{lem:uniquenes:of:direct:factors}
	Let $G=G_1\times\dots\times G_m=H_1\times\dots\times H_n$ be two decompositions of $G$
	into directly indecomposable factors. Then $n=m$ and there is a permutation $\sigma\in S_m$
	such that for all $i$ we have $G_i\cong H_{\sigma(i)}$ and $G_iZ(G)=H_{\sigma(i)}Z(G)$. 
\end{lem}
\begin{proof}
	The first part is the well-known Krull-Remak-Schmidt Theorem and the addition that $G_iZ(G)=H_{\sigma(i)}Z(G)$ can be easily derived (see for example~\cite[Corollary 6]{DBLP:conf/icalp/KayalN09})
\end{proof}

In particular, the collection of subgroups $\{G_iZ(G)\}_{1\leq i\leq m}$ is invariant under automorphisms as a whole. Later we show that the union of these subgroups, i.e.,~$\bigcup_{i=1}^m G_iZ(G)$, is $5$-WL$_\versone$-detectable.
\begin{dfn}
	We say a central decomposition $G=H_1H_2$ is \textbf{directly induced} if there are subgroups $K_i\leq H_i$
	such that $G=K_1\times K_2$ and $H_i=K_iZ(G)$.
\end{dfn}

Whenever there is a pairing between the indecomposable direct factors of two groups, such that each pair is indistinguishable via WL-refinement, then the groups themselves are indistinguishable as well. This is a simple observation in terms of pebble games (given in the next lemma). The other direction, namely that indistinguishable groups always admit such a pairing of indecomposable direct factors, is investigated in the remainder of this section and turns out to be highly non-trivial.

\begin{lem}
	If $J\in\{{\versone},{\verstwo}\}$, $k\geq 3$,
	$G_1\equiv^J_k H_1$ and $G_2\equiv^J_k H_2$, then $G_1\times G_2\equiv^J_k H_1\times H_2$.
\end{lem}
\begin{proof}
	Consider the $(k+1)$-pebble game (Version $J$) on $(G_1\times G_2,H_1\times H_2)$. 
	Assume Duplicator always chooses bijections componentwise,
	$f_1:G_1\to H_1$ and $f_2:G_2\to H_2$ say, and combines them to a move $f:(g,h)\mapsto(f_1(g_1),f_2(g_2))$.
	Given $k$-tuples $((g_{1,1},g_{2,1}),\dots,(g_{1,k},g_{2,k}))\in (G_1\times G_2)^{(k)}$ and
	$((h_{1,1},h_{2,1}),\dots,(h_{1,k},h_{2,k}))\in (H_1\times H_2)^{(k)}$, there is an isomorphism
	mapping $(g_{1,i},g_{2,i})$ to $(h_{1,i},h_{2,i})$ for all $i$ if and only if there is a componentwise isomorphism
	mapping $g_{1,i}$ to $h_{1,i}$ and a componentwise isomorphism mapping $g_{2,i}$ to $h_{2,i}$ for all $i$.
	In particular, Duplicator can choose $f_1$ and $f_2$ according to winning strategies 
	on $(G_1,H_1)$ and $(G_2,H_2)$ and obtain a winning strategy on the direct products.
\end{proof}

\subsection{Abelian and Semi-Abelian Case}
Direct products with abelian groups are easier to handle than the general case and serve as a basis for reduction later on.

\begin{dfn}
An element $x\in G$ \textbf{splits} from the group $G$ if there is a \textbf{complement} $H\leq G$ of $x$ in $G$, i.e., $G=\langle x\rangle\times H$.
\end{dfn} 

\begin{lem}\label{lem:SplitFromDirectAbelian}
	Let $A$ be a finite, abelian $p$-group and consider an arbitrary cyclic decomposition $A=A_1\times\dots\times A_m$. Then $a=(a_1,\dots,a_m)\in A$ splits from $A$ if and only if
	there is some $i$ with $|a|=|a_i|$ and $a_i\in A_i\setminus (A_i)^p$.
\end{lem}
\begin{proof}
	First assume that $|a|=|a_i|$ and $a_i\in A_i\setminus (A_i)^p$ for some $i$.
	Then $A_i=\langle a_i\rangle$ and $A=\langle a\rangle\times \langle \{e_j: j\in \{1,\ldots,m \} \setminus \{i\}\}\rangle$, where $e_j$ is a generator of $\{1\}\times\dots\times\{1\}\times A_j\times \{1\}\times\dots\times\{1\}$. For the other direction, assume that 
	$A=\langle a\rangle\times B$ holds for some subgroup $B\leq A$. Then it also holds that
	$A=\langle a'\rangle\times B$ for every element $a'=ax$ with $|x|<|a|$. So we may assume
	for all $i$ that either $|a_i|=|a|$ or $a_i=1$ holds. If $e_i^{\frac{|e_i|}{p}}\in B$ holds for all $i$
	with $|a_i|=|a|$, then $a^{\frac{|a|}{p}}\in B$ which is a contradiction to $\langle a\rangle\cap B=\{1\}$. So there is some
	$i$ with $|a_i|=|a|$ and $e_i^{\frac{|e_i|}{p}}\not\in B$.
	Hence, $\langle e_i\rangle\cap B=\{1\}$ and then $A=\langle e_i\rangle\times B$
	since $|e_i|\geq |a|$. This finally implies $|a_i|=|a|=|e_i|$ and so $a_i\in A_i\setminus (A_i)^p$.
\end{proof}

\begin{cor}\label{cor:SplitFromDirectAbelian}
	Let $A$ be a finite, abelian $p$-group and $x\in A$. Then $1\neq x$ splits from $A$ if and only if there is no $y\in A$ such that $|xy^p|<|x|$. Moreover, $2$-WL$_{{\verstwo}}$ detects the set of all elements that split from $A$.
	\end{cor}
\begin{proof}
	The first part is a restatement of the previous lemma. Regarding the $2$-WL$_\verstwo$-detectability, consider that by Lemma~\ref{lem:basicClosure}, $\{a^p\mid a\in A\}\leq A$ is detectable and $x$ does not split from $A$ if and only if for all $a\in A$ and it holds
	$(xa^p)^{|x|/p}\neq 1$. Thus, the claim follows from Lemma~\ref{LemGroupExpression}.
\end{proof}

\begin{lem}\label{LemDetectSplittingAbelian}
	Let $A$ be a finite abelian group and $A=P_1\times\dots\times P_m$ the decomposition of $A$ into Sylow-subgroups. Then $1\neq x=(x_1,\dots,x_m)$ splits from $A$ if and only if each $x_i$ is either trivial or splits from $P_i$. 
	In particular, $2$-WL$_{{\verstwo}}$ detects the set of elements that split from an abelian group.
\end{lem}
\begin{proof}
	By the Chinese Remainder Theorem we have $\langle x\rangle\cong\langle x_1\rangle\times\dots\times\langle x_m\rangle$. Regarding the detactability note that $x_i$ splits with respect to $P_i$ if and only if $x^{|A|/|P_i|}$ splits in $P_i$, so the claim follows from the previous lemma and Lemma~\ref{lem:basicClosure}. 
\end{proof}

\begin{lem}\label{LemCharacterizeSplitting}
	Let $G$ be a finite group and $z\in Z(G)$. Then $z$ splits from $G$ if and only if $zG'$ splits from $G/G'$ and $\langle z\rangle\cap G'=\{1\}$. 
\end{lem}
\begin{proof}
	Assume that $G/G'=\langle zG'\rangle\times K$. For a generating set $(k_1G',\dots,k_mG')$ of $K$, let 
	$\hat{K}:=\langle k_1,\dots,k_m\rangle G'\leq G$. Then $G=\hat{K}\langle z\rangle$, and $\langle zG'\rangle\cap K=\{1\}$ together with $\langle z\rangle\cap G'=\{1\}$ implies $\langle z\rangle\cap\hat{K}=\{1\}$. Thus $z$ splits with complement $\hat{K}$. On the other hand, if $G=\langle z\rangle\times H$ then $G'=H'$ so $G/G'=\langle zH'\rangle \times H/H'$.	
\end{proof}

\begin{cor}\label{DetectSplittingElements}
	The set of elements splitting from a finite group $G$ is $4$-WL$_{{\versone}}$-detectable.
\end{cor}
\begin{proof}
	First note that
	$3$-WL$_\verstwo$ detects both $G'$ and $Z(G)$ (Corollary~\ref{cor:derived_detectable} and Lemma~\ref{LemCenterDetectable}) and thus also the set of central elements $z$ with $\langle z\rangle\cap G'=\{1\}$. Furthermore, $G/G'$ is abelian and elements splitting from $G/G'$ can be detected with $2$-WL$_\verstwo$.
	By Lemma~\ref{MainThm1} a), this information can be lifted to $G$ by $k$-WL$_\versone$ for $k\geq 4$, i.e.,
	$4$-WL$_\versone$ detects the set of elements $z\in G$ such that $zG'$ splits from $G'$ as well as
	central elements $z\in G$ with $\langle z\rangle\cap G'=\{1\}$.
\end{proof}

We analyze the splitting of elements in two special instances.

\begin{lem}\label{SplittingInSubgroup}
	Consider groups $U\leq G$ and $x\in Z(G)\cap U$. If $x$ splits from $G$ then $x$ splits from $U$.
\end{lem}
\begin{proof}
	If $x$ splits from $G$ we can write $G=\langle x\rangle\times K$ for a suitable complement
	$K\leq G$ of $x$ in $G$. Then $U:=\langle x^{m_1}k_1,\dots,x^{m_t}k_t\rangle$ where
	$k_1,\dots,k_t\in K$ and since $x\in U$ we obtain $U=\langle x,k_1,\dots,k_t\rangle=
	\langle x\rangle\times \langle k_1,\dots,k_m\rangle$.
\end{proof}

\begin{lem}\label{lem:splitting_direct_componentwise}
	Consider a direct product $G=G_1\times G_2$ and a $p$-element $z:=(z_1,z_2)\in Z(G)$.
	Then $z$ splits from $G$ if and only if $z_i$ splits from $G_i$ for some $i\in\{1,2\}$ which fulfills $|z_i|=|z|$.
\end{lem}
\begin{proof}
	Since $z$ is a $p$-element, so is $z_1\in G_1$ and $z_2\in G_2$.
	First assume that $G=\langle z\rangle\times B$ for some suitable $B\leq G$.
	Without loss of generality assume that $|z_1|=|z|$. If $(z_1,1)\cap B=\{1\}$ then $G=\langle (z_1,1)\rangle\times B$, so $z_1$ splits from $G_1$ by the previous lemma.
	Otherwise $(z_1^m,1)\in B$ for some $m$ such that $z_1^m\neq 1$. By assumption
	it holds that $z^m\notin B$, thus it must be the case that $|z_2|=|z|$ and $\langle(1,z_2)\rangle\cap B=\{1\}$, so $z_2$ splits from $G_2$.
	
	For the other direction, if $G_i$ admits a decomposition $G_i=\langle z_i\rangle\times B_i$ for some $i\in\{1,2\}$ 
	with $|z_i|=|z|$, then $G=\langle (z_1,z_2)\rangle \times \left(B_i\times G_{i+1\mod 2}\right)$.
\end{proof}

Let us move on to the semi-abelian case, by which we mean groups of the form $H\times A$ where $A$ is abelian and $H$ does not have abelian direct factors.
\begin{lem}\label{WLSplitMAximalAbelian}
	Let $G=H\times A$ with $A$ a maximal abelian direct factor. Then the isomorphism type of $A$ is identified by $4$-WL$_{{\versone}}$, i.e., if $\tilde{G}\equiv^\versone_4 G$ then $\tilde{G}$ has a maximal abelian direct factor isomorphic to $A$.
\end{lem}
\begin{proof}
	Consider a prime $p$ that divides $|G|$. If $\tilde{G}$ is another group with $\tilde{G}\equiv^\versone_4 G$ then $|G|=|\tilde{G}|$ and by Lemma~\ref{LemCenterDetectable} we have $Z(\tilde{G})\cong Z(G)$. Since abelian groups are direct products of their Sylow-subgroups, there must be an isomorphism between the respective Sylow-$p$-subgroups of the centers,
	$Z$ and $\tilde{Z}$ say. Write $\tilde{G}=\tilde{H}\times\tilde{A}$ with maximal abelian direct factor $\tilde{A}$. We can decompose $Z$ as $Z=Z_1\times\dots\times Z_m$ with $Z_i\cong C_{p^i}^{e_i}$, $e_i\geq 0$, and for each $i$ there are subgroups $H_i\leq Z(H)$ and $A_i\leq A$ such that $Z_i=H_i\times A_i$. Similarly define $\tilde{Z}_i,\tilde{H}_i$ and $\tilde{A}_i$.	Since $Z\cong \tilde{Z}$ it also holds that $Z_i\cong \tilde{Z}_i$ for all $i$
	and therefore it is enough to show that $|A_i|=|\tilde{A}_i|$ for all $i$. Since $H$ does not admit abelian direct factors, Lemma~\ref{lem:splitting_direct_componentwise} implies that central elements of order $p^i$ split from $G$ if and only if $|A_i(x)|=p^i$, where $A_i(x)$ is the projection of $x$ onto the component $A_i$ in the decomposition of $Z$ from above. 
	The same then holds for $\tilde{G}$ and $\tilde{A}_i$. By Lemma~\ref{DetectSplittingElements}, $4$-WL$_\versone$ detects the set of central elements of order $p^i$ that split from a group. In particular, if $\tilde{G}\equiv^\versone_4 G$ holds then $|\{x\in Z\mid |x|=p^i=|A_i(x)|\}|=|\{x\in \tilde{Z}\mid |x|=p^i=|\tilde{A}_i(x)|\}|$ which in turn shows $|A_i|=|\tilde{A}_i|$, since both these groups are some direct power of $C_{p^i}$ by definition.
\end{proof}

Controlling the non-abelian part is more complicated. We first introduce a new technical framework.
\begin{dfn}\label{dfn:component:wise:filt}
	Let $G=L\times R$. A \textbf{component-wise filtration} of $U\leq G$ w.r.t.~$L$ and $R$ is a chain of subgroups $\{1\}=U_0\leq\dots\leq U_r=U$ such that for all $1\leq i<r$, we have $U_{i+1}
	\leq U_i(L\times\{1\})$ or $U_{i+1}\leq U_i(\{1\}\times R)$. The filtration is $k$-WL$_\versone$-\textbf{detectable} if all subgroups in the chain are.
\end{dfn}

\begin{lem}\label{lem:iso_type_max_abelian_factor}
	Let $G=H\times A$ with maximal abelian direct factor $A$. There exists a component-wise filtration
	of $Z(G)$ with respect to $H$ and $A$, say $\{1\}=U_0\leq\dots\leq U_r=Z(G)$, that is $4$-WL$_\versone$-detectable.
\end{lem}
\begin{proof}
	First let $p_1<\dots <p_n$ denote the primes dividing $|G|$ and write $Z_{p_i}$ for the Sylow-$p_i$ subgroup of $Z(G)=Z(H)\times A$. Assume we already have a component-wise filtration of 
	\[
		U=Z_{p_1}\times\dots\times Z_{p_{i-1}}\times \{ z\in Z_{p_i}\mid |z|< p_i^m\}
	\] 
	with respect to $H$ and $A$ which is furthermore $4$-WL$_\versone$-detectable. We will argue how to extend it to $U\{z\in Z_{p_i}\mid |z|\leq p_i^m\}$ and then the claim follows by induction. To simplify our notation let $p:=p_i$ and let $N$ be maximal such that $p^N$ divides $|Z(G)|$. 
	
	Set $V_0:=\{ z\in Z_{p}\mid |z|< p^m\}$ and for $j\geq 1$ define
	\[
		V_{j}:=\langle \{z^{p^{N-j}}\mid z\in Z_p, |z^{p^{N-j}}|\leq p^{m} \}\rangle V_{j-1},
	\] 
		so we aim to extend the filtration such that elements with roots of higher order are added in earlier steps.
	Further define
	\[
		W_j:=\langle \{ z^{p^{N-j}}\mid z\in Z_p,|z|\leq p^{N-j+m}~\text{and}~z~\text{does not split from}~G\}\rangle V_{j-1}.
	\] 
	By construction, we have 
	\[
		U=UV_0\leq UW_1\leq UV_1\leq\dots\leq UW_N\leq UV_N=U\{ z\in Z_{p_i}\mid |z|\leq p_i^m\}.
	\]It remains to show that all $W_j$ and $V_j$ are detectable in $G$ and that the subchain $V_{j-1}\leq W_{j}\leq V_{j}$ ascends component-wise for all $j\geq 1$. 
	
	To show that $W_j$ and $V_j$ are detectable, recall that the set of elements that split from $G$ are detectable via $4$-WL$_\versone$ according to Corollary~\ref{DetectSplittingElements} and central $e$-th powers are detectable for all $e\in\mathbb{Z}$ according to Lemma~\ref{lem:basicClosure}, thus $V_j$ and $W_j$ are detectable for all $j$.
	
	To show component-wise ascension, note that if $A=\langle a\rangle\times K_a$ then for all $h\in Z(H)$ with $|h|\leq |a|$ it holds
	$G=(H\times K_a)\times\langle (h,a)\rangle$. So if $(h,a)\in Z(H)\times A$ does not split from $G$ then either $|h|>|a|$ or $a$ does not split from $A$ and then there is some $b\in A$ with $|ab^p|<|a|$ according to Corollary~\ref{cor:SplitFromDirectAbelian}. Consider $x:=(h,a)^{p^{N-j}}\in W_j$ where $|(h,a)|=p^{N-j+m}$ and $(h,a)$ does not split from $G$. If $|h|>|a|$ then
	$x\in (h^{p^{N-j}},1)V_{j-1}$ since $V_{j-1}$ contains all $p$-elements of order smaller than $p^m$. Otherwise $|a|=p^{N-j+m}$ and there is some $b\in A$ with $|ab^{-p}|<|a|$. First, this implies $x\in (h,b^{p})^{p^{N-j}}V_{j-1}$, again using the fact that $V_{j-1}$ contains $V_0$. 
	Now by definition $(1,b^{p^{N-j+1}})\in V_{j-1}$ and thus $x\in (h^{p^{N-j}},1)V_{j-1}$.	
	In conclusion, $V_{j-1}\leq W_j$ is a component-wise extension. The same holds for $W_j\leq V_j$, since $H$ has no abelian direct factors and so if $(h,a)\in V_j$ splits from $G$ then also $(1,a)$ splits from $G$ (this follows from Lemma~\ref{lem:splitting_direct_componentwise}) and so it holds that $V_j\leq (\{1\}\times A)W_j$.
\end{proof}

\begin{lem}\label{lem:semi-abelian-case}
	Consider $G:=H\times A$ and $\hat{G}=\hat{H}\times \hat{A}$ where $A$ and $\hat{A}$ are maximal abelian direct factors. Then, for $k\geq 5$, $G\equiv^\versone_{k}\hat{G}$ implies $H\equiv^\versone_{k-1} \hat{H}$.
\end{lem}
\begin{proof}
	By Lemma~\ref{WLSplitMAximalAbelian} we obtain $A\cong \hat{A}$. Consider the component-wise filtrations from the proof of the previous lemma, $1=U_0\leq\dots\leq U_r=Z(G)$ and $1=\hat{U}_0\leq\dots\leq \hat{U}_r=Z(\hat{G})$, with respect to the decompositions $G:=H\times A$ and $\hat{G}=\hat{H}\times \hat{A}$. Then $U_i$ and $\hat{U}_i$ are detectable by $4$-WL$_\versone$ in $G$ and $\hat{G}$, respectively. Since we assume $G\equiv^\versone_{k}\hat{G}$ we can also assume that $U_i$ and $\hat{U}_i$ obtain the same stable colors for all $i$. Furthermore, $G\equiv^\versone_{k}\hat{G}$ implies $U_i\lneq (\{1\}\times A)U_{i+1}$ if and only if $\hat{U}_i\lneq (\{1\}\times \hat{A})\hat{U}_{i+1}$ for all $i$, as well as~$U_i\lneq (H\times \{1\})U_{i+1}$ if and only if~$\hat{U}_i\lneq (\hat{H}\times \{1\})\hat{U}_{i+1}$ for all $i$.
	
	We first show the following claim $\circledast$: For all $1\neq x\in Z(H)\times\{1\}$ and $1\neq y\in \{1\}\times A$ we have
	$\min\{i\mid x\in U_i\}\neq\min\{i\mid y\in U_i\}$. To see this, let $i$ be minimal with $x\in U_i$. By definition of component-wise filtrations and minimality of $i$, there are $h_i\in Z(H)$ with $U_i=\langle U_{i-1},h_1,\dots, h_t\rangle$. In particular, if $y\in U_i$ then $y\in U_{i-1}$ which shows the claim.
	In the same way elements of $Z(H)\times\{1\}$ can be distinguished from those in $\{1\}\times\hat{A}$, since we assume $G\equiv^\versone_{k}\hat{G}$.
	
	We make use of Lemma~\ref{ResepctSubgroupChains} regarding the subgroup chains that are defined by the chosen filtrations. Since $G\equiv^\versone_{k}\hat{G}$ Duplicator has a winning strategy in the ${k+1}$-pebble game on $(G,\hat{G})$ and then, via Lemma~\ref{ResepctSubgroupChains}, Duplicator has a winning strategy in the $k$-pebble game
	 where all bijections $f:G\to\hat{G}$ Duplicator chooses respect the subgroup chains and their respective cosets, i.e. $f(gU_i)=f(g)\hat{U}_i$ for all $i$. Then $\circledast$ implies that whenever $g_1g_2^{-1}\in Z(H)\times\{1\}$ we have $f(g_1)f(g_2)^{-1}\not\in \{1\}\times \hat{A}$.
	
	Next, we show that Duplicator must map $H\times\{1\}$ to a system of representatives modulo $\{1\}\times \hat{A}$ in each move. Otherwise there would be $(h_1,1),(h_2,1)\in G=H\times A$ and $(h,a_1),(h,a_2)\in\hat{G}=\hat{H}\times\hat{A}$ with $f(h_i,1)=(h,a_i)$. Then $(h,a_1)(h,a_2)^{-1}$ is central so the same must hold for $(h_1h_2^{-1},1)$ (since $f$ must in particular fulfill $fg(Z(G))=f(g)Z(G)$ for all $g\in G$) but then the latter is contained in $Z(H)\times\{1\}$ while $(h,a_1)(h,a_2)^{-1}\in\{1\}\times\hat{A}$, a contradiction. 
	
	In particular, this means that Spoiler can restrict the game to $H\times\{1\}$ and if it is the case that $H\not\equiv^\versone_{k-1} \hat{H}$, then Spoiler can ultimately reach a configuration $[((h_1,1),\dots,(h_{k-1},1),\perp),((x_1,a_1),\dots,(x_{k-1},a_{k-1}),\perp)]$ such that the induced configuration over $(G/(1\times A),\hat{G}/(1\times\hat{A}))$ fulfills the winning condition for Spoiler. The only possibility for the original configuration not to fulfill the winning condition for Spoiler is that there exist $i,j,m$ and either $h_i\neq h_j$, $x_i=x_j$ and $a_i\neq a_j$, but then 
	$(h_ih_j^{-1},1)$ is distinguished from $(x_i,a_i)(x_j,a_j)^{-1}=(1,a_ia_j^{-1})$ via $\circledast$,
	or $h_ih_j\neq h_m$, $x_ix_j=x_m$ and $a_ia_j\neq a_m$ in which case $(h_ih_jh_k^{-1},1)$ is distinguished from $(x_i,a_i)(x_j,a_j)(x_m,a_m)^{-1}=(1,a_ia_ja_m^{-1})$ via $\circledast$. Since $k\geq 5$ both cases can be exploited by Spoiler to win the $k$-pebble game on $(G,\hat{G})$ by Lemma~\ref{RespectPartialSubgroups}, which is a contradiction.
\end{proof}

\subsection{General Case}
Building on the previous paragraph, we reduce the general case to the semi-abelian case. Consider an indecomposable direct decomposition $G=G_1\times\dots\times G_d$, then we know that the collection of subgroups $\{G_iZ(G)\mid 1\leq i\leq d\}$ is independent of the chosen decomposition. We first show that $\bigcup_i G_iZ(G)$ can be detected by Weisfeiler-Leman refinement and then we exploit the fact that the non-commuting graph of $G$ induces components on $\bigcup_i G_iZ(G)$ which correspond to the groups $G_iZ(G)$.
\begin{dfn}
	Given a group $G$, we define the \textbf{non-commuting graph} $\Gamma_G$ with vertex set $G$, in which
	two elements $g,h\in G$ are joined by an edge if and only if $[g,h]\neq 1$.
\end{dfn}

\begin{lem}[\cite{Abdollahi_2006}, Prop. 2.1]\label{lem:noncom-connected}
	If $G$ is non-abelian then $\Gamma_G[G\setminus Z(G)]$ is connected.
\end{lem}

We now approximate $\bigcup_i G_iZ(G)$ from below by constructing a canonical central decomposition of $G$ which is WL-detectable.
\begin{dfn}
	Consider a finite, non-abelian group $G$.
	Define $M_1\subseteq G$ to be the set of non-central elements~$g$ whose centralizers~$C_G(g)$ have maximal order among all non-central elements. Iteratively define $M_{i+1}$ by adding those elements~$g$ to $M_i$ that have maximal centralizer order~$|C_G(g)|$ among the remaining elements $G\setminus\langle M_i\rangle$. Set $M:=M_{\infty}$ to be the stable set resulting from this process. 
	Consider the subgraph of $\Gamma_G$ induced on $M$ and let $K_1\dots,K_m$ be its connected components.
	Set $N_i:=\langle K_i\rangle$. We call $N_1,\dots,N_m$ the \textbf{non-abelian components} 
	of $G$. 
\end{dfn}

\begin{lem}\label{PropertiesNonAbelianComponents}
	In the notation of the previous definition, the following hold:
	\begin{enumerate}
		\item $M$ is detectable in $G$ by $3$-WL$_{{\verstwo}}$
		\item $G=N_1\cdots N_m$ is a central decomposition of $G$.
		For all $i$, $Z(G)\leq N_i$ and $N_i$ is non-abelian. In particular $M$ generates $G$.
		\item If $G=G_1\times\dots\times G_d$ is an arbitrary direct decomposition, then for each $1\leq i\leq m$ there is exactly
		one $1\leq j\leq d$ with $N_i\subseteq G_jZ(G)$. Collect all such $i$ for one fixed $j$ in an index set~$I_j$.
		Then the product over all $N_i$ for $i\in I_j$ is equal to $G_jZ(G)$.
	\end{enumerate}
\end{lem}
\begin{proof}
	\begin{enumerate}
		\item $M_1$ is detectable by $2$-WL$_\verstwo$ since group elements are generally distinguishable by the orders of their centralizers (Lemma~\ref{lem:basicClosure}). Assume that $M_i$ is detectable by $3$-WL$_\verstwo$.
		Then $\langle M_i\rangle$ is detectable as well by Lemma~\ref{lem:basicClosure} and so is $G\setminus\langle M_i\rangle$.
		Thus, elements of $G\setminus\langle M_i\rangle$ are distinguishable from all other elements and can then be further distinguished by the orders of their centralizers. So $M_{i+1}$ is detectable and the claim follows inductively.
		\item By definition, the construction of $M$ does not terminate until $M_i$ contains a generating set of $G$, so
		$G=\langle M\rangle$. For $g\in G$ and $z\in Z(G)$ it holds that $C_G(gz)=C_G(g)$,
		thus $M_1Z(G)=M_1$ and then $M_iZ(G)=M_i$ via induction. 
		
		We claim that the connected components
		of $\Gamma_G$ induced on $M$ all contain more than one element. Otherwise say $K_i=\{x\}$
		and so $[x,M]=\{1\}$. Since $\langle M\rangle =G$, it follows that $x$ must be central. But by construction we never add central elements to $M$. 
		In conclusion, $\Gamma_G[M]$ is a disjoint
		union of non-trivial components. So if $x\in K_i$, there is some $y\in K_i$ with $xy\neq yx$ and hence, $N_i$ is non-abelian for all $i$. For all $z\in Z(G)$ we also have $(xz)y\neq y(xz)$ and conclude that~$xz \in K_i$. Overall we obtain 
		$K_iZ(G)=K_i$, implying that $Z(G)\leq N_i$. Finally we note that by definition of $\Gamma_G$, $N_i$ and $N_j$ centralize each other for $i\neq j$. 
		\item We first argue that all elements in $M$ belong to some $G_jZ(G)$. Assume otherwise
		that $x=(x_1,\dots,x_d)\in M$ with~$x_i\in G_i$ and more than one $x_i$ is non-central. Then
		$x=(x_1,1\dots,1)\cdot (1,x_2,1,\dots,1)\dots (1,\dots,1,x_d)$ is a product of elements, each
		with a strictly bigger centralizer than $x$, and so $x$ would have never been selected to be added to 			$M$. Thus, each element of $M$ belongs to exactly one $G_jZ(G)$ and if two elements from $M$
		do not commute they must belong to the same $G_jZ(G)$. 
		Finally assume that $\prod_{I_j} N_i=H_j \lneq G_jZ(G)$. Since all $N_i$ contain $Z(G)$, there must be some non-central element $x$ in $G_jZ(G)\setminus H_j$. But then $x$ is also not contained in $H_j\times (\times_{i\neq j} G_i)Z(G)$ contradicting the fact that $G=N_1\dots N_m$.\qedhere
	\end{enumerate}
\end{proof}

\begin{dfn}
	Let $G=N_1\cdots N_m$ be the decomposition into non-abelian components and let $G=G_1\times\dots\times G_d$ be an arbitrary direct decomposition. We say $x\in G$ is \textbf{full for $(G_{j_1},\dots,G_{j_r})$}, if $\{ 1\leq i\leq m\mid [x,N_i]\neq 1\}=I_{j_1}\cup\dots\cup I_{j_r}$. 
	For all $x\in G$ define $C_x:=\Pi_{[x,N_i]=\{1\}} N_i$ and $N_x:=\Pi_{[x,N_i]\neq \{1\}} N_i$.
\end{dfn}

\begin{obs}\label{obs:existence_full_elts}
	Given an arbitrary collection of indices $J\subseteq\{1,\dots,m\}$, the group elements $x\in G$ that have $C_x=\Pi_{i\in J}N_i$ are exactly those elements of the form $x=z\Pi_{i\in J}n_i$ with $z\in Z(G)$ and $n_i\in N_i\setminus Z(G)$.
	In particular, full elements exist for every collection of non-abelian direct factors and any direct decomposition and they are exactly given by products over non-central elements from
	the corresponding non-abelian components.
\end{obs}

\begin{lem}\label{DirectlyInducedFull} Let $G$ be non-abelian and let $G=G_1\times\dots\times G_d$ be an indecomposable direct decomposition. 
	For all $x\in G$ we have a central decomposition $G=C_xN_x$ with $Z(G)\leq C_x$ and~$Z(G)\leq N_x$. The decomposition is directly induced if and only if $x$ is full for a collection of direct factors of $G$.
\end{lem}
\begin{proof}
	If $x$ is full for $(G_{j_1},\dots,G_{j_r})$ then by Lemma~\ref{PropertiesNonAbelianComponents}, 
	$N_x=(G_{j_1}\times\dots\times G_{j_r})Z(G)$ and $C_x=(\times_{i\notin\{j_1,\dots,j_r\}}G_i)Z(G)$, so the central decomposition $G=C_xN_x$ is directly induced.

	For the other direction assume the decomposition $G=C_xN_x$ is directly induced and consider subgroups $\tilde{C}_x\leq C_x$, $\tilde{N}_x\leq N_x$ with $G=\tilde{C}_x\times\tilde{N}_x$, $C_x=\tilde{C}_xZ(G)$
	and $N_x=\tilde{N}_xZ(G)$. Consider indecomposable direct decompositions $\tilde{C}_x=\times_{t\in I_C} \tilde{C}_t$ and
	$\tilde{N}_x=\times_{t\in I_N} \tilde{N}_t$, where $I_C$ and $I_N$ are suitable index sets. Then, for all $t\in I_C$, the group $\tilde{C}_t$ is an indecomposable direct factor of $G$, so $\tilde{C}_tZ(G)=G_{j_t}Z(G)$ for some $j_t$ by Lemma~\ref{lem:uniquenes:of:direct:factors}. In particular, $[x,N_i]=1$ for all $i\in I_{j_t}$. Similarly, for $\tilde{N}_t$ and $t\in I_N$,
	$[x,N_i]\neq 1$ for all $i\in I_{j_t}$. Since each $G_jZ(G)$ is either contained in $C_x$ or $N_x$, the claim follows.
\end{proof}

\begin{lem}\label{DetectDirectlyIndcued}
	Suppose $G=G_1\times G_2$. For $k\geq 4$ assume that $k$-WL$_{{\verstwo}}$ detects
	$G_1Z(G)$ and detects $G_2Z(G)$ and does not distinguish $G$ from some other group $H$. Then for~$i\in\{1,2\}$ there
	are subgroups $H_i\leq H$ with $\chi^{\verstwo,k}_G(G_iZ(G))=\chi^{\verstwo,k}_H(H_iZ(H))$ and 
	$H_iZ(H)\equiv^{{\verstwo}}_k G_iZ(G)$ such that
	$H=H_1\times H_2$.
\end{lem}

\begin{proof}Set $\tilde{G}_i:=G_iZ(G)$. As a consequence of Lemma~\ref{lem:basicClosure}, there exist subgroups of $H$, $\tilde{H}_i$ say, that correspond to $\tilde{G}_i$ with respect to stable color classes of $k$-WL$_\verstwo$. It is also implied that $Z(H)\leq \tilde{H}_i$.
	Consider the decompositions $Z(G)=Z(G_1)\times Z(G_2)$ and $\tilde{G}_i=G_i\times Z(G_{i+1\bmod 2})$ and observe the following: If $x$ splits from $Z(G)$ then, using Lemma~\ref{lem:splitting_direct_componentwise}, we see that $x$ also splits from $\tilde{G}_1$ or
	$\tilde{G}_2$. The observation is used to prove Claim 1 below.
	Write $\tilde{H}_i:=R_i\times B_i$ where $B_i$ is a maximal abelian direct factor of $\tilde{H}_i$.
	\paragraph{Claim 1:} For all possible choices of $R_i$ and $B_i$ it holds that $R_1\cap R_2= \{1\}$.
	
	By assumption, $\tilde{H}_1\cap\tilde{H}_2=Z(H)$ so $R_1\cap R_2\leq Z(H)$. For the sake of contradiction assume that there exists $z\in R_1\cap R_2$ such that $|z|=p$ for some prime $p$. Then there also exists a central $p$-element $w$ that splits form $Z(H)$ such that $z\in\langle w\rangle$ (this is always true for central elements of prime order, as we can take $w$ to be a root of $z$ of highest $p$-power order in the abelian group $Z(H)$). Write $w=(r_i,b_i)$ with respect to the chosen decompositions of $\tilde{H}_i$. For some $m\in\mathbb{N}$ we have that $w^m=z\in R_1\cap R_2$, so $1\neq w^m=(r_1^m,1)=(r_2^m,1)$, in particular
	$|b_i|<|r_i|$ for $i=1,2$, since $w$ has $p$-power order. Then $w$ does not split from $\tilde{H}_i$ or otherwise, by Lemma~\ref{lem:splitting_direct_componentwise}, $r_i$ would split from $R_i$ but $R_1$ and $R_2$ do not admit abelian direct factors. 
	
	Hence, $w$ splits from $Z(H)$ but not from $\tilde{H}_1$ or $\tilde{H}_2$ and such elements do not exist with respect to $G$, $\tilde{G}_1$ and $\tilde{G}_2$ as pointed out above.
	Claim 1 follows, since $k$-WL$_\verstwo$ detects the set of elements splitting from $Z(H)$, $\tilde{H}_1$ or $\tilde{H}_2$, respectively.  \hfill$\tiny\blacksquare$.
	
	Next we consider maximal abelian direct factors $A$ and $B$ of $G$ and $H$, respectively. Write $H:=R\times B$. By Lemma~\ref{WLSplitMAximalAbelian}, $A\cong B$.
	\paragraph{Claim 2:} $R_1$ and $R_2$ can be chosen such that $R_1R_2\cap B=\{1\}$.
	
	Let $\tilde{H}_1=\langle (r_1,b_1),\dots,(r_t,b_t)\rangle\leq R\times B$ then, since $B\leq\tilde{H}_1$,
	\[
		\tilde{H}_1=\langle (r_1,1)(1,b_1),\dots,(r_t,1),(1,b_t)\rangle=\langle (r_1,1)\dots (r_t,1)\rangle\times B.
	\] An analogous statement holds for $\tilde{H}_2$, so $R_1R_2$ can be chosen as a subgroup of $R$.\hfill$\tiny\blacksquare$
	
	To finally prove the Lemma, let $R_1$ and $R_2$ be as in Claim 2. By Claim 1, $R_1\cap R_2=\{1\}$, so $R_1R_2B=R_1\times R_2\times B\leq H$. By Lemma~\ref{WLSplitMAximalAbelian}, $G\equiv^{\verstwo}_k H$ implies
 that $|R_1||R_2||B|=|H|$, so $H=R_1\times R_2\times B$ and this can be written as $(R_1\times B_1)\times (R_2\times B_2)$, where $B_i\leq H_i$ are chosen such that $B_1\times B_2=B$ and $B_i$ is isomorphic to a maximal abelian direct factor of $G_i$.
 Furthermore, $R_iZ(H)=\tilde{H}_i$ by construction.
\end{proof}

\begin{lem}
	Let $G=N_1\cdots N_m$ be the decomposition into non-abelian components and 
	$G=G_1\times\dots\times G_d$ a decomposition into indecomposable direct factors. 
	For $k\geq 5$, $k$-WL$_{{\verstwo}}$ detects the set of elements
	that are full for only one $G_i$ as well as the pairs of elements that are full
	for the same collection of direct factors.
\end{lem}
\begin{proof}
	Let $x,y\in G$ and assume $x$ is full for $G_{j_1},\dots,G_{j_r}$.
	Using Lemma~\ref{AlgoVSGame}, we consider the $(k+1)$-pebble game with initial configuration
	$[(x,\perp^k),(y,\perp^k)]$. As long as there is a pebble pair on $(x,y)$, by Lemma~\ref{PropertiesNonAbelianComponents}, Duplicator has to map $C_x$ to $C_y$ and $N_x$ to $N_y$. If $x$ is not distinguished from $y$ then by Lemma~\ref{DetectDirectlyIndcued}, the central decomposition $G=C_yN_y$ has to be directly induced, since the same holds for $G=C_xN_x$ and $k\geq 5$. By Lemma~\ref{DirectlyInducedFull} the element $y$ is full as well.
	
	So let $F\subseteq G$ be the set of elements that are full for some collection of direct factors. We just showed that $F$ is detectable by $5$-WL$_\verstwo$. Note that for $x,y\in F$ it is easily detectable via WL-refinement that $x$ and $y$ are full for the same collection of direct factors since this is equivalent to $C_x=C_y$. It remains to show that elements that are full for only one direct factor can be distinguished from the rest of $F$. This follows from the fact that
	$x\in F$ is full for a single direct factor if and only if $C_x$ is minimal with respect to inclusion among $C_f$, $f\in F$ and this can be detected with $k\geq 3$.
\end{proof}

\begin{cor}
	If $G=G_1\times\dots\times G_d$ is a decomposition into
	indecomposable direct factors then $\bigcup_i G_iZ(G)$ is
	detected in $G$ by $5$-WL$_{{\verstwo}}$.
\end{cor}
\begin{proof}
	For $k\geq 5$, by the previous result $k$-WL distinguishes elements that are full for one fixed direct factor from other elements. Call the set of full elements $\mathcal{F}$.
	Then for each $g\in\mathcal{F}$ it holds that $N_g$ is of the form $G_iZ(G)$ for
	some $i$ and by Observation~\ref{obs:existence_full_elts} each $i$ occurs through some
	full element of $G$.
	Thus, via Lemma~\ref{lem:basicClosure}, the union $\bigcup_{g\in\mathcal{F}}N_g=\bigcup_i G_iZ(G)$ is $5$-WL$_\verstwo$-detectable.
\end{proof}

\begin{thm}
	Let $G=G_1\times\dots\times G_d$ be a decomposition into
	indecomposable direct factors and $k\geq 5$. If $G\equiv^{{\verstwo}}_k H$ then there are
	indecomposable direct factors $H_i\leq H$ such that $H=H_1\times\dots\times H_d$ and $G_i\equiv^{\verstwo}_{k-1} H_i$ for all $i$. Moreover
	$G$ and $H$ have isomorphic maximal abelian direct factors and
	$G_iZ(G)\equiv^{{\verstwo}}_k H_iZ(H)$. 
\end{thm}
\begin{proof}
	Since $\mathcal{F}_G:=\bigcup_i  G_iZ(G)$ is $5$-WL$_\verstwo$-detectable, the group $H$ must be decomposable into indecomposable direct factors $H=\times_j H_j$ such that $\mathcal{F}_H=\bigcup_j  H_jZ(H)\subseteq H$ is indistinguishable from $\mathcal{F}_G$. Consider the non-commuting graphs of $G$ and $H$ induced on these sets and recall that non-commuting graphs of non-abelian groups are connected (Lemma~\ref{lem:noncom-connected}). Since different direct factors in a fixed decomposition centralize each other, we obtain that for each non-singleton connected component $K$ of $\Gamma_G[\mathcal{F}_G]$ there exists a unique indecomposable direct factor $G_i$ such that $K=G_iZ(G)\setminus Z(G)$ and thus $\langle K\rangle=G_iZ(G)$. Again by Lemma~\ref{lem:noncom-connected}, all non-abelian direct factors appear in this way.
	
	The same holds for $H$ and so if $G$ is not distinguishable from $H$, there must be a bijection between the components of $\Gamma_G[\mathcal{F}_G]$ and $\Gamma_H[\mathcal{F}_H]$, such that the subgroups generated by corresponding components are indistinguishable via $5$-WL$_\verstwo$. This defines a correspondence $G_iZ(G)\equiv^{\verstwo}_k H_iZ(H)$ after reordering the factors of $H$ in an appropriate way. 
	From Lemma~\ref{lem:semi-abelian-case} it follows that $G_i\equiv^{\verstwo}_{k-1} H_i$.
	By Lemma~\ref{lem:iso_type_max_abelian_factor}, $G$ and $H$ must have isomorphic maximal abelian direct factors, so for abelian factors we even have $G_i\cong H_i$.
\end{proof}

\section{Conclusion}
We studied the Weisfeiler-Leman dimension of numerous isomorphism invariants of groups, showing that a low dimensional WL algorithm in fact captures a plethora of isomorphism invariants, characteristic subgroups, and group properties classic to algorithmic group theory. 
Particularly tricky was the treatment of direct indecomposable factors, for which we had to circumvent the fact that the they do not correspond to canonical substructures of the groups. Our techniques lead us to a canonical maximal central decomposition.

The observation that many efficiently computable isomorphism invariants are captured by a low dimensional WL algorithm raises the question whether there are actually invariants that are not captured at all.
Here we should emphasize that it is an open problem whether some fixed dimension of WL represents a complete invariant. The question is equivalent to the well-known open question whether the Weisfeiler-Leman dimension of groups is bounded in general (stated explicitly in~\cite{WLonGroups}).

For this open question, our results show that it suffices to consider directly indecomposable groups. We wonder whether there are other, similar reductions to confine the search for groups of high WL-dimension.

\bibliographystyle{plain}
\bibliography{refs}

\end{document}